\documentclass[12pt,reqno,a4paper]{amsart}

\usepackage{mdframed}
\usepackage{setspace}
\usepackage{mathrsfs}  
\usepackage{comment}

\usepackage[latin1]{inputenc}

\usepackage[linkcolor=red,colorlinks=true]{hyperref}

\usepackage{bold-extra}
\usepackage{etoolbox}
\usepackage{amsmath}
\usepackage{enumerate}
\usepackage{amssymb}
\usepackage{amscd}
\usepackage{amsthm}
\usepackage{amsfonts}
\usepackage{graphicx}
\usepackage[all,cmtip]{xy}

\patchcmd{\subsection}{-.5em}{.5em}{}{}
\patchcmd{\subsubsection}{-.5em}{.5em}{}{}

\makeatother
\usepackage{hyperref}

\numberwithin{equation}{section}

\newcommand{\cA}{\mathcal{A}}

\newcommand{\cC}{\mathcal{C}}

\newcommand{\cL}{\mathcal{L}}

\newcommand{\cP}{\mathcal{P}}

\newcommand{\cR}{\mathcal{R}}

\newcommand{\bN}{\mathbb{N}}

\newcommand{\bP}{\mathbb{P}}

\newcommand{\bR}{\mathbb{R}}

\newcommand{\bZ}{\mathbb{Z}}

\newcommand{\ra}{\rightarrow}

\newcommand{\qor}{\quad \textrm{or} \quad}

\newcommand{\qand}{\quad \textrm{and} \quad}

\newcommand\subsetsim{\mathrel{%
\ooalign{\raise0.2ex\hbox{$\subset$}\cr\hidewidth\raise-0.8ex\hbox{\scalebox{0.9}{$\sim$}}\hidewidth\cr}}}
\newcommand{\eps}{\varepsilon}

\DeclareMathOperator{\Hom}{Hom}

\DeclareMathOperator{\pr}{pr}

\DeclareMathOperator{\supp}{supp}

\DeclareMathOperator{\Prob}{Prob}

\DeclareMathOperator{\Z}{Z}
\DeclareMathOperator{\R}{\bR}

\renewcommand{\phi}{\varphi}

\definecolor{lichtgrijs}{gray}{0.95}
\mdfdefinestyle{mystyle}{ %
    backgroundcolor=lichtgrijs, %
    linewidth=0pt, %
    innertopmargin=10pt, %
    innerbottommargin=10pt,%
    nobreak=true
}

\theoremstyle{theorem}
\newtheorem{theorem}{Theorem}[section]
\newtheorem{corollary}[theorem]{Corollary}
\newtheorem{proposition}[theorem]{Proposition}
\newtheorem{lemma}[theorem]{Lemma}

\theoremstyle{task}

\theoremstyle{definition}
\newtheorem{definition}[theorem]{Definition}
\newtheorem{remark}[theorem]{Remark}
\newtheorem{construction}[theorem]{Construction}

\newtheorem{example}[theorem]{Example}

\newtheorem{setting}[theorem]{Setting}

\usepackage{changepage}
\usepackage{mathtools}
\usepackage{graphicx}
\usepackage{calc}

\usepackage[toc,page]{appendix}

\usepackage{scalerel}

\usepackage{extarrows}

\DeclareMathSymbol{\shortminus}{\mathbin}{AMSa}{"39}
\usepackage{mathtools}
\begin{document}
\bibliographystyle{plain} 

\title[Hulls of quasimorphisms]{Dynamics on spaces of quasimorphisms and applications to approximate lattice theory}
\author{Michael Bj\"orklund}
\address{Department of Mathematics, Chalmers, Gothenburg, Sweden}
\email{micbjo@chalmers.se}
\thanks{}

\author{Tobias Hartnick}
\address{Institut f\"ur Algebra und Geometrie, KIT, Karlsruhe, Germany}
\curraddr{}
\email{tobias.hartnick@kit.edu}
\thanks{}

\begin{abstract} We study the dynamics of countable groups on their respective spaces of quasimorphisms. For cohomologically non-trivial quasimorphisms we show that there are no invariant measures and classify stationary measures. Within the equivalence class of any given quasimorphism we find both uniquely stationary orbit closures which are in fact boundaries and orbit closures with uncountably many ergodic stationary probability measures. We apply these results to study hulls of uniform approximate lattices which arise from twists by quasimorphisms. We show that these hulls do not admit invariant probability measures (extending results by Machado and Hrushovski) and classify stationary probability measures on these hulls.
\end{abstract}
\maketitle

\section{Introduction}
This article is concerned with the classification of stationary measures on spaces of quasimorphisms. It is motivated from a problem in the theory of approximate lattices, and our results have applications to Delone dynamical systems arising in this context.

\subsection{Uniform approximate lattices without invariant measures}
To state our original motivation we recall that a subset $\Lambda$ of a locally compact group $G$ is called an \emph{approximate subgroup} \cite{Tao} if it is symmetric, contains the identity and satisfies $\Lambda \Lambda \subset \Lambda F$ for some finite subset $F \subset G$. It is called a \emph{uniform approximate lattice} \cite{BH18} if it is moreover a Delone set in $G$, i.e.\ uniformly discrete and relatively dense. Typical examples of uniform approximate lattices include uniform lattices as well as a large class of mathematical quasicrystals known as Meyer set.

With every uniform lattice $\Gamma \subset G$ one can associate a dynamical system $G \curvearrowright G/\Gamma$, and this system admits a unique $G$-invariant probability measure. A natural generalization of this dynamical system to a uniform approximate lattice $\Lambda \subset G$ is the so called \emph{hull system} $G \curvearrowright \Omega_{\Lambda}$, where $\Omega_\Lambda$ denotes the orbit closure of $\Lambda$ in the Chabauty space of closed subsets of $G$. By analogy with the lattice case we asked in \cite{BH18} whether the hull of a uniform approximate lattice in $G$ always admits a $G$-invariant probability measure.

A negative answer to this question follows from recent work of Hrushovski and Machado. More precisely, Hrushovski provided in \cite[Section 7.9]{H20} an explicit example of a uniform approximate lattice which is non-laminar, and it was later shown by Machado \cite{Machado} that the hull of a non-laminar uniform approximate lattice cannot admit an invariant probability measure. While this provides examples of uniform approximate lattices without invariant measures, the proof of this fact in \cite{Machado} is very indirect. 

In the present article we provide further examples of uniform approximate lattices whose hulls do not admit invariant probability measures, based on the following construction.
\begin{definition}
Let $P \subset \R$ be a closed subset, let $\Gamma$ be a countable group and let $\phi: \Gamma \to \bR$ be a real-valued quasimorphism. Then the \emph{$\phi$-twist} of $P$ is defined as the subset of $\Gamma \times \bR$ given by
\[
P(\phi) := \{(\gamma, t) \in \Gamma \times \bR \mid \phi(\gamma) + t \in P\}.
\]
\end{definition}
By elementary ergodic theoretic methods we can show:
\begin{proposition}\label{PropIntroA} Assume that $P \subset \R$ is closed and that $\R$ acts freely on $\Omega_P$. If $\phi: \Gamma \to \R$ is a quasimorphism with $\phi(e) = 0$ which is not at bounded distance from a homomorphism, then the orbit closure $\Omega_{P(\phi)}$ does not admit a $(\Gamma \times \R)$-invariant probability measure.
\end{proposition}
It is easy to see that if $P \subset \R$ happens to be a uniform approximate lattice and $\phi: \Gamma \to \Z$ is integer-valued and anti-symmetric, then $P(\phi)$ is a uniform approximate lattice in $\Gamma \times \bR$ (see Corollary \ref{GetUAL}). This produces plenty of examples of uniform approximate lattices whose hulls do not support invariant probability measure and motivates a more detailed study of 
dynamical systems of the form $\Gamma \times \R \curvearrowright \Omega_{P(\phi)}$.

\subsection{Stationary measures on hulls of twist sets}\label{StaMeas}
For the following results, let $\Gamma$ be a countable group, let $P \subset \bR$ be a closed subset such that the action of $\R$ on $\Omega_{\R}$ is free and uniquely ergodic. Moreover, let $\phi: \Gamma \to \bR$ be a quasimorphism with $\phi(e) = 0$, which is not at bounded distance from a homomorphism. 

We fix an absolutely continuous symmetric probability measure $p$ on $\Gamma \times \R$ such that $\phi$ is $p$-integrable. We then denote by $p_\Gamma$ the pushforwards of $p$ to $\Gamma$ and consider the space $\Prob_p(\Omega_{P(\phi)})$ of $p$-stationary probability measures on $\Omega_{P(\phi)}$. We write $\mathcal L_p(\phi)$ for the space of left-$p$-harmonic quasimorphisms which are at bounded distance from $\phi$. 

\begin{theorem}\label{ThmB} If $\phi$ is left-$p$-harmonic, then $\Omega_{P(\phi)}$ admits a unique $p$-stationary probability measure $\nu$ and $(\Omega_{P(\phi)}, \nu)$ is a relatively measure-preserving extension over a non-trivial $(\Gamma, p_\Gamma)$-boundary.
\end{theorem}
\begin{theorem}\label{ThmC} If $\psi\in \mathcal L_{p}(\phi)$ and $\nu_\psi$ denotes the unique $p$-stationary probability measure on $\Omega_{P(\psi)}$, then for every $\nu \in \Prob_p(\Omega_{P(\phi)})$ there exists a measure-preserving $\nu$-measurable map $(\Omega_{P(\phi)}, \nu) \to (\Omega_{P(\psi)}, \nu_\psi)$. 
\end{theorem}
\begin{corollary} If $\nu \in \Prob_p(\Omega_{P(\phi)})$, then $(\Omega_{P(\phi)}, \nu)$ admits a measurable $(\Gamma \times \bR)$-factor, which is a non-trivial $(\Gamma, p_\Gamma)$-boundary.
\end{corollary}
\begin{theorem}\label{ThmD} There exists a quasimorphism $\phi'$ at bounded distance from $\phi$ such that $\Omega_{P(\phi')}$ admits uncountably many ergodic $p$-stationary probability measures.
\end{theorem}
If $\Gamma$ surjects onto $\bZ$, then the quasimorphism $\phi'$ in Theorem \ref{ThmD} can be chosen to be antisymmetric and to take integer values. This implies:
\begin{theorem}\label{ThmE} If $\Gamma$ surjects onto $\bZ$, then there exists a uniform approximate lattice in $\Gamma \times \R$ which admits uncountably many ergodic $p$-stationary probabiliy measures, but no $(\Gamma \times \R)$-invariant probability measure. In particular, there exists such a uniform approximate lattice in $\mathrm{GL}_2(\R)$.
\end{theorem}
All of these results follow readily from corresponding results about dynamics of quasimorphisms, which we now explain.
\subsection{Relation to dynamics of quasimorphisms}
Let $\Gamma$ be a countable group and equip $\mathscr F_o(\Gamma) := \{\phi \in \R^{\Gamma} \mid \phi(e) = 0\}$ with the product topology. Then $\Gamma$ acts continuously on $\mathscr F_o(\Gamma)$ via
\begin{equation}\label{ActionOfGamma}
(\gamma.\varphi)(\gamma') = \varphi(\gamma' \gamma) - \varphi(\gamma), \quad \gamma,\gamma' \in \Gamma, \enskip \varphi \in \mathscr{F}_o(\Gamma).
\end{equation}
We now fix a quasimorphism $\phi_o \in \mathscr F_o(\Gamma)$ which is not at bounded distance from a homomorphism. We then denote by $[\phi_o]$ the set of all quasimorphisms $\phi \in \mathscr F_o(\Gamma)$ which are at bounded distances from $\phi_o$ and by $\Delta_{\phi_o}$ the orbit closure of $\phi_o$ in $\mathscr F_o(\Gamma)$. We then have the following elementary observation:
\begin{proposition}\label{PropTrivial}
$\Delta_{\phi_o}$ is a compact metrizable subset of $[\phi_o]$, which does not admit any $\Gamma$-invariant probability measure.
\end{proposition}
Proposition \ref{PropIntroA} is thus an immediate consequence of the following proposition; here given a $\Gamma$-space $\Delta$, an $\bR$-space $\Omega$ and a continuous cocyle $c: \Gamma \times \Delta\to \R$ we denote by $\Delta \times_c \Omega$ the skew product $\Delta \times \Omega$ with $(\Gamma \times \bR)$-action given by
\[
(\gamma, t)(\phi, \omega) := (\gamma.\phi, \omega+t+c(\gamma, \phi)) \quad (\gamma \in \Gamma, t \in \bR, \phi \in \Delta, \omega \in \Omega).
\]
\begin{proposition}\label{PropA} If $P \subset \bR$ is a closed subset such that $\R$ acts freely on $\Omega_P$, then $\Omega_{P(\phi_o)} \cong \Delta_{\phi_o} \times_c \Omega_P$, where $c(\gamma, \phi)=\phi(\gamma)$. In particular, $\Delta_{\phi_o}$ is a continuous factor of $\Omega_{P(\phi_o)}$.
\end{proposition}
Concerning stationary measures we have the following result; here $p$ is as in Subsection \ref{StaMeas}.
\begin{proposition}\label{PropB} Let $P \subset \bR$ be a closed subset such that the $\R$-action on $\Omega_P$ is free and uniquely ergodic with invariant measure $\theta$. Then the map
\[
\Prob_{p}(\Delta_{\phi_o}) \to \Prob_p(\Delta_{\phi_o} \times_c \Omega_P), \quad \nu \mapsto \nu \otimes \theta
\] 
is bijective. In particular, $\Prob_p(\Omega_{P(\phi_o)}) \cong \Prob_{p}(\Delta_{\phi_o})$.
\end{proposition}
To establish the results from Subsection \ref{StaMeas} we thus need to analyze the $\Gamma$-action on $\Delta_{\phi_o}$. This study is also of independent interest in the theory of quasimorphisms and will occupy most of the present article.

\subsection{The left-harmonic core of a quasimorphism}
Let $\Gamma$ be a countable group and let $\phi_o: \Gamma \to \bR$ be a quasimorphism with $\phi_o(e) = 0$. We fix a symmetric probability measure $p$ on $\Gamma$ whose support generates $\Gamma$ and such that $\phi_o$ is $p$-integrable. We recall that $\mathcal L_p(\phi_o) \subset [\phi_o]$ denote the subspace of left-$p$-harmonic quasimorphisms; note that this space is invariant under $\Gamma$.
\begin{definition} The \emph{left-$p$-harmonic core} of $\phi_o$ is the subset
\[
\cC_p(\phi_o) := \bigcap_{\phi \in \cL_p(\phi_o)} \Delta_{\phi_o} \subset \cL_p(\phi_o) \subset [\phi_o].
\]
\end{definition}
The following theorem implies that $\cC_p(\phi_o)$ is non-empty; it is the main technical result of this article and implies Theorem \ref{ThmB} and Theorem \ref{ThmC}.
\begin{theorem}[Stationary measures and left-harmonic core]\label{ThmX} If the quasimorphism $\phi_o: \Gamma \to \bR$ is not at bounded distance from a homomorphism, then the following hold:
\begin{enumerate}[(i)]
\item Every $p$-stationary measure on $\cL_p(\phi_o)$ is supported on $\cC_p(\phi_o)$.
\vspace{0.1cm}
\item $\cC_p(\phi_o)$ admits a unique $p$-stationary probability measure $\nu_{\phi_o}$. \vspace{0.1cm}
\item $(\cC_p(\phi_o), \nu_{\phi_o})$ is a non-trivial $(\Gamma, p)$-boundary. \vspace{0.1cm}
\item If $\nu$ is any $p$-stationary measure on $\Delta_{\phi_o}$, then there exists a measure-preserving $\nu$-measurable map $(\Delta_{\phi_o}, \nu) \to (\cC_p(\phi_o), \nu_{\phi_o})$.
\end{enumerate}
\end{theorem}
We complement this by the following result, which implies both Theorem \ref{ThmD} and Theorem \ref{ThmE}.
\begin{theorem}\label{ThmF} Every quasimorphism $\phi_o: \Gamma \to \R$ is equivalent to a quasimorphism $\phi_1: \Gamma \to \bZ$ such that $\Delta_{\phi_1}$ admits uncountably many $p$-stationary and $\Gamma$-ergodic probability measures. If $\Gamma$ surject onto $\Z$, then $\phi_1$ can be chosen to be symmetric.
\end{theorem}
\subsection{Organization of the article}
This article is organized as follows: In Section \ref{Boundary} we recall some background on random walks on countable groups and their Poisson boundaries. Section \ref{SecQuasi} discusses hulls of quasimorphisms and establishes Proposition \ref{PropTrivial}. 

Our results concerning stationary measures on spaces of quasimorphisms are achieved in Sections \ref{SecCore} -- \ref{SecMany}. More precisely,  Section \ref{SecCore} establishes the first three parts of Theorem \ref{ThmX}, whereas the final part is established in Sections \ref{SecCore2}. Section \ref{SecMany} establishes Theorem \ref{ThmF}.

Section \ref{SecTwisted} discusses twisted sets and their hulls. We first establish Proposition \ref{PropA} and Proposition \ref{PropB}. In view of these propositions, Proposition \ref{PropIntroA} then follows from Proposition \ref{PropTrivial}, Theorem \ref{ThmB} and Theorem \ref{ThmC} follow from Theorem \ref{ThmX}, and Theorem \ref{ThmD} and Theorem \ref{ThmE} follow from Theorem \ref{ThmF}.\\

\textbf{Acknowledgment.} This article arose from discussions during the Workshop on ``Aperiodic order and approximate groups''  as KIT during August 2023. We thank KIT for the funding of the workshop and are grateful to the participants for their input. M.B.\ was supported by Grant 11253320 from the Swedish Research Council (VR)

\section{Preliminaries on boundary theory}\label{Boundary}
Throughout this section let let $\Gamma$ be a countable group and let $p \in \mathrm{Prob}$ be a symmetric probability measure on $\Gamma$ whose support generates $\Gamma$; we then refer to $(\Gamma, p)$ as a \emph{symmetric measured group}. We denote by $\bP_p = p^{\otimes \bN}$ the induced measure on $\Gamma^{\bN}$; the
sequence $(Z_n)$ of $\Gamma$-valued random variables given by
\[
Z_n: (\Gamma^{\bN}, \bP_p) \to \Gamma, \quad (\omega_n)_{n \in \bN} \mapsto \omega_1 \cdots \omega_n.
\]
is called the associated \emph{random walk}. 

A function $f: \Gamma \to \bR$ such that $\gamma \mapsto f(\gamma \gamma')$ (respectively $\gamma \mapsto f(\gamma' \gamma)$ is $p$-integrable for every $\gamma' \in \Gamma$ is called \emph{left-$p$-harmonic} (respectively \emph{right-$p$-harmonic}) if $p \ast f = f$ (respectively $f \ast p = f$). We denote by $\mathscr{H}_l^\infty(\Gamma, p)$ (respectively $\mathscr{H}_r^\infty(\Gamma, p)$ the spaces of all bounded left-$p$-harmonic (respectively bounded right-$p$-harmonic) functions. If $B$ is a standard Borel space with an action of $\Gamma$ by Borel automorphisms, then a Borel probability measure $m$ on $B$ is called \emph{$p$-stationary} if $p \ast m = m$; in this case we call $(B,m)$ a \emph{$(\Gamma, p)$-space}. We say that a $(\Gamma, p)$-space $(B,m)$ is \emph{ergodic} if $m$ is $\Gamma$-ergodic.

Given a standard Borel space $B$ with an action of $\Gamma$ by Borel automorphisms, we denote by $\mathrm{Prob}_p(B)$ the space of $p$-stationary Borel probability measures on $B$. We write $\Prob_p^{\textrm{erg}}(B)$ and $\mathrm{Prob}_\Gamma(B)$ for the subspaces of all $\Gamma$-ergodic, respectively $\Gamma$-invariant measures in  $\mathrm{Prob}_p(B)$. For the following classical statement see e.g.\ \cite[Theorem 2.13]{Furman}:
\begin{lemma}
\label{Lemma_Poisson}
For every symmetric measured group $(\Gamma,p)$ there exists an ergodic $(\Gamma, p)$-space $(B,m)$ such that the following hold:
\begin{itemize}
\item[(i)] The Poisson transform $\mathscr{P}_m : L^\infty(B,m) \ra \mathscr{H}^\infty_r(\Gamma,p)$ defined by
\[
\mathscr{P}_m f(\gamma) = \int_B f(\gamma.b) \, dm(b) \quad (\gamma \in \Gamma)
\]
is a surjective isometry. \vspace{0.1cm}
\item[(ii)] For every countable subset $\cA \subset L^\infty(B,m)$, there exist a $\bP_p$-conull Borel set $\Omega \subset \Gamma^{\bN}$ and a Borel map $Z_\infty : \Omega \ra B$ such that
\[
(Z_\infty)_*\bP_p|_\Omega = m \qand \lim_{n \ra \infty} \mathscr{P}_m f(Z_n(\omega)) = f(Z_\infty(\omega)),
\]
for all $\omega \in \Omega$ and $f \in \cA$.
\end{itemize}
Furthermore, $(B,m)$ is unique up to measure-preserving $\Gamma$-isomorphisms.
\end{lemma}
The space $(B,m)$ is called the \emph{Poisson boundary} of the measured group $(\Gamma, p)$. If $(B,m)$ is the Poisson boundary of $(\Gamma, p)$ and $\pi: B \to B'$ is a $\Gamma$-equivariant Borel map, then $(B', \pi_*m)$ is called a $(\Gamma, p)$-boundary.
Note that every $(\Gamma,p)$-boundary is $\Gamma$-ergodic. Finally, a $(\Gamma, p)$-boundary $(B', m')$ is called \emph{trivial} if $m' = \delta_{b'}$ for some $\Gamma$-fixed point $b' \in B'$, and \emph{non-trivial} otherwise.

\section{Hulls of quasimorphisms}\label{SecQuasi}
Throughout this section, let $\Gamma$ be a countable group.
\subsection{Quasimorphisms}
\begin{definition}
The \emph{defect set} of a function $\varphi: \Gamma \to \bR$ is defined as
\[
 \mathscr{D}(\varphi) := \{\phi(gh)-\phi(g)-\phi(h) \mid g,h \in \Gamma\}.
\]
A function $\varphi$ is called a \emph{quasimorphism} if $\mathscr{D}(\varphi)$ is bounded; in this case the number
\[
D(\varphi) := \sup_{g,h \in \Gamma} |\phi(gh)-\phi(g)-\phi(h)|
\]
is called the \emph{defect} of $\varphi$.
\end{definition}
In this article we will only consider quasimorphisms which are \emph{normalized} in the sense that $\varphi(e) = 0$. Two quasimorphisms $\varphi_1, \varphi_2$  with $\|\varphi_1-\varphi_2\|_\infty < \infty$ are said to be \emph{equivalent}. Every quasimorphism $\varphi$ is equivalent to a $\bZ$-valued quasimorphism, e.g.\ $\lfloor \varphi \rfloor$, hence to a quasimorphism with a \emph{finite} defect set. A quasimorphism is called \emph{cohomologically trivial} if it is equivalent to a homomorphism. 

\begin{remark}[Antisymmetric quasimorphism]\label{QMBasics} 
We say that a quasimorphism is \emph{antisymmetric} if $\phi(\gamma^{-1}) = -\phi(\gamma)$. If $\phi_o: \Gamma \to \bR$ is a quasimorphism and $A \subset \bR$ is any relatively dense subset, then we can always find an antisymmetric quasimorphism $\phi: \Gamma \to \bR$ which takes values in $A-A$ and is equivalent to $\phi_o$. For example, if $A = \bZ$ we may take
\[
\phi: \Gamma \to \bZ, \quad \phi(\gamma) :=  \lfloor \phi_o(\gamma)/2 \rfloor - \lfloor \phi_o(\gamma^{-1})/2 \rfloor.
\]
\end{remark}
\subsection{Hulls and invariant measures}
We equip $\bR^\Gamma$ with the product topology and consider the continuous action of $\Gamma$  on
\[
\mathscr{F}_o(\Gamma) := \{\varphi \in \bR^\Gamma \mid \varphi(e) = 0\}.
\]
given by \eqref{ActionOfGamma}. The orbit closure of $\varphi_o \in \mathscr{F}_o(\Gamma)$ will be denoted by $\Delta_{\varphi_o}$ and referred to as the \emph{hull} of $\varphi_o$. By definition of the product topology, point evaluation defines a continuous map $c: \Gamma \times \mathscr{F}_o(\Gamma) \to \bR$, $(h, \varphi) \mapsto \varphi(h)$. This map is a 
\emph{normalized $\Gamma$-cocycle} on $\mathscr{F}_o(\Gamma)$, i.e.\ for all $\varphi \in \mathscr{F}_o(\Gamma)$ and $\gamma_1, \gamma_2 \in \Gamma$ we have
\begin{equation}\label{CanonicalCocycle}
c(\gamma_1\gamma_2, \varphi)  = c(\gamma_1,\gamma_2.\varphi) + c(\gamma_2,\varphi) \text{ and } c(e, \varphi) = 0.
\end{equation}
\begin{proposition}\label{QMTrivialities} Let $\varphi_o: \Gamma \to \bR$ be a normalized quasimorphism.
\begin{enumerate}[(i)]
\item The hull $\Delta_{\varphi_o}$ is a compact metrizable space. \vspace{0.1cm}
\item For all $\varphi \in \Delta_{\varphi_o}$ we have $\|\varphi-\varphi_o\|_\infty \leq D(\varphi_o)$. In particular, the hull $\Delta_{\varphi_o}$ consists of quasimorphisms which are equivalent to $\varphi_o$. \vspace{0.1cm}
\item If $\varphi_o$ is cohomologically non-trivial, then $\mathrm{Prob}_\Gamma(\Delta_{\varphi_o}) = \emptyset$.
\end{enumerate}
\end{proposition}
\begin{proof} Pick $\varphi \in \Delta_{\varphi_o}$ and choose a sequence $(\gamma_n)$ in $\Gamma$
such that $\gamma_n.\varphi_o \ra \varphi$ in $\mathscr{F}_o(\Gamma)$. For all $n \in \bN$ and $\gamma \in \Gamma$ we obtain
\[
\|(\gamma_n.\varphi_o)(\gamma) - \varphi_o(\gamma)\| = \|\varphi_o(\gamma\gamma_n) - \varphi_o(\gamma_n) - \varphi_o(\gamma)\| \leq D(\varphi_o),
\]
which implies (ii) and shows that 
\begin{equation}
\label{inclusion}
\Delta_{\varphi_o} \subseteq (\varphi_o + \overline{B}(0,{D(\varphi_o)})^{\Gamma}) \cap \mathscr{F}_o(\Gamma).
\end{equation}
It thus follows from Tychonoff's theorem that $\Delta_{\varphi_o}$ is relatively compact (and hence compact and metrizable) with respect to the product topology. This proves (i).

\item Now assume that $\Delta_{\varphi_o}$ admits a $\Gamma$-invariant probability measure $\mu$. By \eqref{inclusion} the integral
\[
\psi(\gamma) = \int_{\Delta_{\varphi_o}} \varphi(\gamma) \, d\mu(\varphi)
\]
converges and satisfies $\|\psi - \varphi_o\|_\infty \leq D(\varphi_o)$ for every $\gamma \in \Gamma$. Moreover, invariance of $\mu$ yields
\begin{align*}
\psi(\gamma_1\gamma_2) 
& = \int_{\Delta_{\varphi_o}} \varphi(\gamma_1 \gamma_2) \, d\mu(\varphi) = \int_{\Delta_{\varphi_o}} \gamma_2.\varphi(\gamma_1) +\varphi(\gamma_2) \, d\mu(\varphi)
\\[0.2cm]
&= \int_{\Delta_{\varphi_o}} \varphi(\gamma_1) \, d\mu(\varphi) + \int_{\Delta_{\varphi_o}} \varphi(\gamma_2) \, d\mu(\varphi) = \psi(\gamma_1) + \psi(\gamma_2),
\end{align*}
for all $\gamma_1, \gamma_2 \in \Gamma$, which shows that $\psi \in \Hom(\Gamma, \bR)$ and hence $\varphi$ is cohomologically trivial. We thus obtain (iii) by contraposition.
\end{proof}
\subsection{Harmonic quasimorphisms}
From now on let $(\Gamma, p)$ be a symmetric measured group and let $\phi_o: \Gamma \to \bR$ be a $p$-integrable quasimorphism. Note that integrability holds automatically if $p$ is finitely supported or if $\Gamma$ is finitely-generated and the word length with respect to some finite generating set is $p$-integrable. We will write \[
\mathcal L_{p}(\varphi_o) = \{\psi \in \mathscr{F}_o(\Gamma) \mid p \ast \psi = \psi, \; \|\psi-\varphi_o\|_\infty < \infty\}\]
for the space of left-harmonic quasimorphisms at bounded distance from $\varphi_o$ and $\mathcal R_{p}(\varphi_o)$ for the space of right-harmonic quasimorphisms at bounded distance from $\varphi_o$. To show that these spaces are non-empty we can make use of a $p$-stationary probability measure on $\Delta_{\varphi_o}$. Note that such measures exist in view of Proposition \ref{QMTrivialities}.(i). 
\begin{proposition}\label{ProduceHarmonic} For every  $\nu \in \Prob_p(\Delta_{\varphi_o})$ the function $\psi_\nu: \Gamma \to \bR$ given by
\[
\psi_\nu(\gamma) := \int_{\Delta_{\varphi_o}} \varphi(\gamma) \, d\nu(\varphi), \quad \gamma \in \Gamma.
\]
satisfies $\psi_\nu \in \cR_p(\varphi_o)$. In particular, $\cR_p(\varphi_o) \neq \emptyset$.
\end{proposition}
\begin{proof} By Proposition \ref{QMTrivialities}.(ii) we have $\|\psi_\nu - \phi_o\|_\infty \leq D(\phi_o)$. Moreover, since  since $p$ is symmetric and $\nu$ is $p$-stationary, we have
\begin{align*}
(\psi_\nu * p)(\gamma)
&= \sum_{\gamma' \in \Gamma} p(\gamma')  \psi_\nu(\gamma \gamma') = \sum_{\gamma \in \Gamma} p^{*n}(\gamma) \int_{\Delta_{\varphi_o}} \varphi(\gamma) \, \mathrm{d}\nu(\varphi) 
\sum_{\gamma' \in \Gamma} p(\gamma') \int_{\Delta_{\varphi_o}} \varphi(\gamma \gamma') \, \mathrm{d}\nu(\varphi) \\[0.2cm]
&=
\sum_{\gamma' \in \Gamma} p(\gamma') \int_{\Delta_{\varphi_o}} ((\gamma'.\varphi)(\gamma) + \varphi(\gamma')) \, \mathrm{d}\nu(\varphi)\\
&=  \int_{\Delta_{\varphi_o}} \varphi(\gamma) \, \mathrm{d}\nu(\varphi) +  \sum_{\gamma \in \Gamma} p^{*n}(\gamma) \int_{\Delta_{\varphi_o}} \varphi(\gamma) \, \mathrm{d}\nu(\varphi) 
\end{align*}
for all $\gamma \in \Gamma$. If we abbreviate
\[
\ell_\nu(n) := \sum_{\gamma \in \Gamma} p^{*n}(\gamma) \int_{\Delta_{\varphi_o}} \varphi(\gamma) \, \mathrm{d}\nu(\varphi) \quad (n \in \bN),
\]
then this can be written as
\[
\psi_\nu * p = \psi_\nu + \ell_\nu(1),
\]
and we thus have to show that $\ell_\nu(1) = 0$. Note that for all $m,n \in \bN$ we have
\begin{align*}
\ell_\nu(m+n) 
&= \sum_{\gamma_1, \gamma_2 \in \Gamma} p^{*m}(\gamma_1) p^{*n}(\gamma_2) \int_{\Delta_{\varphi_o}} \varphi(\gamma_1 \gamma_2) \, d\nu(\varphi) \\[0.2cm]
&=
\sum_{\gamma_1, \gamma_2 \in \Gamma} p^{*m}(\gamma_1) p^{*n}(\gamma_2) \int_{\Delta_{\varphi_o}} (\gamma_2.\varphi(\gamma_1) + \varphi(\gamma_2)) \, d\nu(\varphi) \\[0.2cm]
&= \ell_\nu(m) + \ell_\nu(n)
\end{align*}
by $p$-stationarity of $\nu$. In particular, $\ell_\nu(n) = n \ell_\nu(1)$ and thus it remains to show only that the sequence $(\ell_\nu(n))$ is bounded. 

By Remark \ref{QMBasics} there exists an antisymmetric quasimorphism $\eta_o$ which is at bounded distance from $\varphi_o$. By Proposition \ref{QMTrivialities}.(ii) there then exists a constant $M>0$ such that $\|\varphi-\eta_o\|_\infty \leq M$ for all $\varphi \in \Delta_{\varphi_o}$. Moreover, $\eta_o$ is $p$-integrable, since $\varphi_o$ is. Since $\eta_o$ is antisymmetric and $p$ is symmetric, we have
\[
\sum_{\gamma \in \Gamma} p^{*n}(\gamma) \eta_o(\gamma) = 
\sum_{\gamma \in \Gamma} p^{*n}(\gamma^{-1}) \eta_o(\gamma^{-1}) = - \sum_{\gamma \in \Gamma} p^{*n}(\gamma) \eta_o(\gamma),
\]
hence $\sum_{\gamma \in \Gamma} p^{*n}(\gamma) \eta_o(\gamma) = 0$ for all $n$. We conclude that
\begin{align*}
|\ell_\nu(n)|
&= \left| \sum_{\gamma \in \Gamma} p^{*n}(\gamma) \int_{\Delta_{\varphi_o}} \varphi(\gamma) \, d\nu(\varphi) \right| \\[0.2cm]
&= \left| \sum_{\gamma \in \Gamma} p^{*n}(\gamma) \eta_o(\gamma) + \sum_{\gamma \in \Gamma}p^{*n}(\gamma) \int_{\Delta_{\varphi_o}} (\varphi - \eta_o)(\gamma) \, d\nu(\varphi) \right| \leq M \end{align*}
so $(\ell_\nu(n))$ is bounded. This shows that $\ell_\nu(1) = 0$ and finishes the proof.
\end{proof}
\begin{remark}[Bi-harmonic quasimorphisms] \label{Rmk.BiHarm} We observe that since $p$ is symmetric, a quasimorphism $\phi_o$ is right-$p$-harmonic if and only if $\check\phi_o$ is left-$p$-harmonic. We deduce that also $\mathcal L_p(\varphi_o) \neq \emptyset$. If we choose $\psi_o \in \mathcal L_p(\varphi_o)$ and $\nu \in \Prob_p(\Delta_{\psi_o})$, then applying the construction of Proposition \ref{ProduceHarmonic} to $\psi_o$ instead of $\varphi_o$ produces a quasimorphism $\psi_1 \in \mathcal L_p(\varphi_o) \cap  \mathcal R_p(\varphi_o)$. This shows that  $\mathcal L_p(\varphi_o) \cap  \mathcal R_p(\varphi_o) \neq \emptyset$. Using double ergodicity of the Poisson boundary one can in fact show that $\mathcal L_p(\varphi_o) \cap  \mathcal R_p(\varphi_o) = \{\psi_1\}$ for some antisymmetric quasimorphism $\psi_1$ (\cite[Proposition 2.3]{BH11}), but we will not need this fact here.
\end{remark}

\section{Stationary measures on spaces of left-harmonic quasimorphisms}\label{SecCore}
\subsection{The left-harmonic core}
Let $(\Gamma, p)$ be a symmetric measured group and let $\phi_o: \Gamma \to \bR$ be a $p$-integrable quasimorphism. In this section we are going to classify $p$-stationary measures on the space $\cL_p(\varphi_o)$ of left-$p$-harmonic quasimorphisms equivalent to $\varphi_o$. Our key tool is the following space:
\begin{construction}[Left-harmonic core]\label{LeftHarmonic} If $\psi_o \in \mathcal L_p(\varphi_o)$, then\footnote{Note that these statements are not true for $\mathcal R_p(\varphi_o)$ instead of $\mathcal L_p(\varphi_o)$.}  $\gamma.\psi_o$ is left-harmonic for every $\gamma \in \Gamma$. It follows with Proposition \ref{QMTrivialities}.(ii) that $\mathcal L_p(\varphi_o)$ is $\Gamma$-invariant and $\Delta_{\psi_o} \subset \mathcal L_p(\varphi_o)$ for all $\psi_o \in \mathcal L_p(\varphi_o)$. For every $\psi_o \in \mathcal L_p(\varphi_o)$ we thus have an inclusion of $\Gamma$-spaces
\[
\mathcal C_p(\varphi_o):=  \bigcap_{\psi \in \mathcal L_p(\varphi_o)} \Delta_{\psi} \subset \Delta_{\psi_o} \subset \mathcal L_p(\varphi_o).
\]
We refer to $\mathcal C_p(\varphi_o)$ as the \emph{left-$p$-harmonic core} of $\varphi_o$.
\end{construction}
\begin{theorem}\label{QuasiMain}
Let  $\mathcal C_p(\varphi_o) \subset \mathcal L_p(\varphi_o)$ denote the left-$p$-harmonic core of $\varphi_o$.
\begin{enumerate}[(i)]
\item Every $p$-stationary probability measure on $\mathcal L_p(\varphi_o)$ is supported on $\mathcal C_p(\varphi_o)$. \vspace{0.1cm}
\item $\mathcal C_p(\varphi_o)$ is compact, non-empty and admits a unique $p$-stationary measure $\nu_{\varphi_o}$. \vspace{0.1cm}
\item $(\mathcal L_p(\varphi_o), \nu_{\varphi_o})$ is a $(\Gamma, p)$-boundary, which is non-trivial if $\psi_o$ is cohomologically non-trivial.
\end{enumerate}
\end{theorem}
The remainder of this section is devoted to the proof of Theorem \ref{QuasiMain}.
\subsection{Ergodic joinings of stationary measures}
We keep the notation of Theorem \ref{QuasiMain}. Moreover, we fix $\psi_o^{(1)}, \psi_o^{(2)} \in \cL_p(\phi_o)$. We abbreviate $\Delta_i := \Delta_{ \psi_o^{(i)} }$. 
and denote by $\mathrm{pr_i}: \Delta_1 \times \Delta_2 \to \Delta_i$ the coordinate projections. Since $\Delta_1$ and $\Delta_2$
are compact we can find $\nu_i \in \mathrm{Prob}_p(\Delta_i)$. For any such measures $\nu_1, \nu_2$ the space
\[
M := \{\eta \in \Prob(\Delta_1 \times \Delta_2) \mid \mathrm{pr_1}_* \eta = \nu_1 \text{ and } \mathrm{pr_2}_* \eta = \nu_2\}
\]
is weak*-compact, convex and $p$-invariant, hence admits a $p$-invariant element, which we may assume to be $\Gamma$-ergodic. We refer to such a probability measure $\eta$ as an \emph{ergodic joining} of $\nu_1$ and $\nu_2$.
\begin{lemma}\label{EtaJoining} If $\nu_i \in \mathrm{Prob}_p(\Delta_i)$ and $\eta$ is an ergodic joining of $\nu_1$ and $\nu_2$, then
for $\eta$-almost every $(\psi_1, \psi_2) \in \Delta_{1} \times \Delta_{2}$ we have $\psi_1 = \psi_2$. 
\end{lemma}
\begin{proof} We first observe that if $\psi_i \in \Delta_i$ for $i \in \{1,2\}$, then $\psi_1,\psi_2$ (and hence $\psi_1-\psi_2$) are left-$p$-harmonic (see Construction \ref{LeftHarmonic}) and $\|\psi_1-\psi_2\|_\infty \leq M$ for some uniform constant $M$ by Proposition \ref{QMTrivialities}.(ii). In particular, the continuous cocycle $\beta: \Gamma \times (\Delta_1 \times \Delta_2) \to \bR$ given by $\beta(\gamma, (\psi_1, \psi_2)) := \psi_1(\gamma)-\psi_2(\gamma)$ is uniformly bounded. By \cite[Theorem 3.1]{Zimmer} (and since $\bR$ has no non-trivial compact subgroups) it is thus equivalent to the trivial cocycle, i.e.\ there exists a Borel function $u: \Delta_1 \times \Delta_2 \to \bR$ such that 
\[
\beta(\gamma, (\psi_1, \psi_2)) = u(\psi_1, \psi_2) - u(\gamma.(\psi_1, \psi_2))
\]
for $\eta$-almost all $(\psi_1, \psi_2)$ and all $\gamma \in \Gamma$. For $M_1>M_o>0$ we consider the sets
\[
A := \{x \in \Delta_1 \times \Delta_2 \mid |u(x)| < M_o\} \qand B:= \{x \in \Delta_1 \times \Delta_2 \mid |u(x)|> M_1\}.
\]
Then $\eta(A) >0$ for all sufficiently large $M_o$, and if $u$ is not essentially bounded, then also $\eta(B)>0$ for all $M_1 > M_o$. In this case the ergodic theorem for stationary actions \cite[Lemma 4.1]{B18} yields
\[
\frac{1}{N} \sum_{n=1}^N \sum_{\gamma \in \Gamma} p^{\ast n}(\gamma) \eta(A \cap \gamma^{-1}.B) \to \eta(A)\eta(B) >0,
\]
hence there exists some $\gamma \in \Gamma$ with $A \cap \gamma^{-1}.B \neq \emptyset$. For any $(\psi_1, \psi_2) \in A \cap \gamma^{-1}.B$ we then have
\[
|\beta(\gamma, (\psi_1, \psi_2))|  = |u(\psi_1, \psi_2) - u(\gamma.(\psi_1, \psi_2))| > M_1-M_o,
\]
contradicting boundedness of $\beta$. Thus $u \in L^\infty(\Delta_1 \times \Delta_2, \eta)$ is essentially bounded. On the other hand, for $\eta$-almost all $(\psi_1, \psi_2)$ we have (by left-$p$-harmonicity of $\psi_1-\psi_2$)
\begin{align*}
0 &= \sum_{\gamma \in \Gamma}p(\gamma)(\psi_1-\psi_2)(\gamma)  =  \sum_{\gamma \in \Gamma}p(\gamma) \beta(\gamma, (\psi_1, \psi_2))\\[0.2cm]
&=  \sum_{\gamma \in \Gamma}p(\gamma)(u(\psi_1, \psi_2) - u(\gamma.(\psi_1, \psi_2))  = u(\psi_1, \psi_2) - p\ast u(\psi_1, \psi_2),
\end{align*}
i.e. $u$ is an essentially bounded left-$p$-harmonic function. By 
\cite[Lemma 3.3]{B17}, this implies that $u$ is essentially $\Gamma$-invariant, and since $\eta$ is ergodic, it follows that $u$ is essentially constant. This implies that $\beta(\gamma, (\psi_1, \psi_2)) = 0$ for $\eta$-almost all $(\psi_1, \psi_2)$, and thus $\psi_1(\gamma) = \psi_2(\gamma)$ for all $\gamma \in \Gamma$ and for $\eta$-almost all $(\psi_1,\psi_2)$.
\end{proof}

\subsection{Unique stationarity}
From Lemma \ref{EtaJoining} we can now easily deduce Parts (i) and (ii) of Theorem \ref{QuasiMain}: 

\begin{proof}[Proof of Theorem \ref{QuasiMain}, Parts (i) and (ii)]
Firstly, if $\psi_o \in \cL_p(\phi_o)$, then $\Delta_{\psi_o}$ is compact, hence admits a $p$-stationary probability measure $\nu_{\psi_o}$. If $\nu$ is another $p$-stationary probability measure, then there exists a joining $\eta$ of $\nu_{\psi_o}$ and $\nu$ as in Lemma \ref{EtaJoining}, and hence $\eta$ is the diagonal self-joining of $\nu = \nu_{\psi_o}$. This shows that 
$\mathrm{Prob}_p(\Delta_{\psi_o}) = \{\nu_{\psi_o}\}$. \\

Now let  $\psi_1, \psi_2 \in \cL_p(\phi_o)$ with hulls $\Delta_1$, $\Delta_2$ and unique stationary measures $\nu_1$ and $\nu_2$. Let $\eta$ be the corresponding joining as in Lemma \ref{EtaJoining}. Then for $\nu_1$-almost every $\psi_1 \in \Delta_1$ we have $\psi_1 = \psi_2$ for some $\psi_2\in \Delta_2$, and hence $\nu_1(\Delta_1 \cap \Delta_2) = 1$. In particular, $\Delta_1 \cap \Delta_2 \neq \emptyset$. Using induction and compactness of hulls we deduce that in fact $\mathcal C_p(\varphi_o) \neq \emptyset$ and $\nu_1(\mathcal C_p(\varphi_o))= 1$. Thus, if we set $\nu_{\varphi_o} := \nu_1|_{\mathcal C_p(\varphi_o)}$, then this defines a stationary probability measure on $\mathcal C_p(\varphi_o)$ and in particular a stationary probability measure on $\Delta_{\psi_o}$ for every $\psi_o \in \cL_p(\phi_o)$. Since $|\mathrm{Prob}_p(\Delta_{\psi_o})| = 1$, this measure is actually the unique $p$-stationary probability measure on $\Delta_{\psi_o}$; this establishes Parts (i) and (ii) of Theorem \ref{QuasiMain}.
\end{proof}

\subsection{Right-harmonic quasimorphisms and boundaries}
We now turn to the proof of Part (iii) of Theorem \ref{QuasiMain}. In view of Part (ii) of the theorem it remains to show only that for \emph{some} $\psi_o \in \mathcal L_p(\phi_o)$ the pair $(\Delta_{\psi_o}, \nu_{\varphi_o})$, where $\nu_{\varphi_o}$ is the unique $p$-stationary probability measure on $\Delta_{\psi_o}$, is a $p$-boundary and that this $p$-boundary is non-trivial if $\psi_o$ is cohomologically non-trivial. We will choose $\psi_o \in \mathcal L_p(\phi_o) \cap  \mathcal R_p(\phi_o)$, which is possible by Remark \ref{Rmk.BiHarm}. Part (iii) of Theorem \ref{QuasiMain} is then a special case of the following general statement concerning \emph{right}-harmonic quasimorphisms:
\begin{lemma}\label{RightHarmonic} If $\psi_o \in \mathcal R_p(\varphi_o)$, then there exists $\nu_{\psi_o} \in \mathrm{Prob}_p(\Delta_{\psi_o})$ such that $(\Delta_{\psi_o}, \nu_{\psi_o})$ is a $p$-boundary. If $\psi_o$ is not cohomologically trivial, then this boundary is non-trivial.
\end{lemma}
We now work towards the proof of Lemma \ref{RightHarmonic}. For this we fix $\psi_o \in \mathcal R_p(\varphi_o)$, and denote by $(B,m)$ the Poisson boundary of the symmetric measured group $(\Gamma,p)$ (as in Lemma \ref{Lemma_Poisson}). 

\begin{lemma}\label{CocycleAlphaRight} There exists a Borel function $\alpha: \Gamma \times B \to \R$ with the following properties:
\begin{enumerate}[(i)]
\item $\alpha$ is a (strict) Borel cocycle, i.e. 
\[
\alpha(\gamma_1\gamma_2, b)  = \alpha(\gamma_1,\gamma_2.b) + \alpha(\gamma_2,b) 
\quad \textrm{for all $\gamma_1, \gamma_2 \in \Gamma$ and $b \in B$},
\]
and 
\[
\psi_o(\gamma) = \int_B \alpha(\gamma, b) \; \mathrm{d}m(b), \quad \textrm{for all $\gamma \in \Gamma$}.
\]
\item The map $\pi_\alpha : B \ra \mathscr{F}_o(\Gamma)$ given by $\pi_\alpha(b)(\gamma) := \alpha(\gamma,b)$ is a $\Gamma$-equivariant Borel map.
\item $\|\pi_\alpha(b) - \psi_o\|_\infty \leq D(\psi_o)$ for all $b \in B$. In particular, for every $b \in B$ the function $\pi_\alpha(b)$ is a quasimorphism on $\Gamma$ which
is equivalent to $\psi_o$.

\end{enumerate}
\end{lemma}
\begin{proof}  
For $\gamma \in \Gamma$, define $\widetilde{f}_\gamma \in \ell^\infty(\Gamma)$  by $\widetilde{f}_\gamma(\gamma') := \psi_o(\gamma\gamma') - \psi_o(\gamma')$. Note that 
\[
\|\widetilde{f}_\gamma\|_\infty \leq D(\psi_o) + |\psi_o(\gamma)|, \quad \textrm{
for all $\gamma \in \Gamma$}.
\]
Since $\psi_o$ is right-$p$-harmonic, $\widetilde{f}_\gamma \in \mathscr{H}^\infty_r(\Gamma,p)$, so by Lemma \ref{Lemma_Poisson} (i), there exists a unique $f_\gamma \in L^\infty(B,m)$ such that $\|f_\gamma\|_\infty = \|\widetilde{f}_\gamma\|_\infty$ and $\widetilde{f}_\gamma = \mathscr{P}_m f_\gamma$. For each $\gamma$, choose a bounded Borel representative $g_\gamma$ such that $f_\gamma = g_\gamma$ $m$-almost everywhere and 
\[
|g_\gamma(b)| \leq \|f_\gamma\|_\infty, \quad \textrm{for all $b \in B$ and $\gamma \in \Gamma$}.
\]
Define $\beta : \Gamma \times B \ra \bR$ by $\beta(\gamma,b) = g_\gamma(b)$. We claim that for all $\gamma_1, \gamma_2 \in \Gamma$, 
\begin{equation}
\label{betaid}
\beta(\gamma_1 \gamma_2,b) = \beta(\gamma_1,\gamma_2.b) + \beta(\gamma_2,b), \quad \textrm{for $m$-almost every $b \in B$}.
\end{equation}
If \eqref{betaid} holds, then, by \cite[Proposition B.9]{ZBook}, there exists a Borel map 
$\alpha : \Gamma \times B \ra \bR$ such that for every $\gamma \in \Gamma$, we have
$\alpha(\gamma,b) = \beta(\gamma,b)$ for all $m$-almost every $b \in B$, and 
\[
\alpha(\gamma_1 \gamma_2,b) = \alpha(\gamma_1,\gamma_2.b) + \alpha(\gamma_2,b), \quad \textrm{for every $b \in B$}.
\]
In particular, 
\begin{align*}
\int_B \alpha(\gamma,b) \, dm(b) 
&= \int_B \beta(\gamma,b) \, dm(b)
= \int_B g_\gamma(b) \, dm(b) \\[0.2cm]
&= \int_B f_\gamma(b) \, dm(b) = \mathscr{P}_m f_\gamma(e) \\[0.2cm]
&=\widetilde{f}_\gamma(e) = \psi_o(\gamma), \quad \textrm{for all $\gamma \in \Gamma$}.
\end{align*}
This proves (i), modulo \eqref{betaid}. To prove this identity, we note since the Poisson transform $\mathscr{P}_m$ is a surjective isometry, it is enough to show
that for all $\gamma_1, \gamma_2 \in \Gamma$, we have
\[
\mathscr{P}_m \beta(\gamma_1\gamma_2,\gamma \cdot)
= \mathscr{P}_m \beta(\gamma_1,\gamma_2\gamma \cdot) + 
\mathscr{P}_m \beta(\gamma_2,\gamma \cdot), \quad \textrm{for all $\gamma \in \Gamma$},
\]
or, equivalently, since $g_\gamma = f_\gamma$ $m$-almost everywhere, 
\[
\widetilde{f}_{\gamma_1 \gamma_2}(\gamma) = \widetilde{f}_{\gamma_1}(\gamma_2 \gamma) +  
\widetilde{f}_{\gamma_2}(\gamma), \quad \textrm{for all $\gamma, \gamma_1, \gamma_2 \in \Gamma$}.
\]
Unwrapping definitions, we see that this amounts to proving 
\[
\psi(\gamma_1\gamma_2\gamma) - \psi(\gamma) = (\psi(\gamma_1\gamma_2\gamma)-\psi(\gamma_2\gamma)) + (\psi(\gamma_2\gamma)-\psi(\gamma)),
\]
for all $\gamma, \gamma_1, \gamma_2 \in \Gamma$, which is trivial, and thus \eqref{betaid} follows. \\

(ii) For all $\gamma, \gamma' \in \Gamma$ and $b \in B$ we have, using that $\alpha$
is a strict cocycle, 
\begin{align*}
(\pi_\alpha(\gamma.b))(\gamma') 
&= \alpha(\gamma',\gamma.b) = \alpha(\gamma' \gamma,b) - \alpha(\gamma,b)
= (\gamma.\pi_\alpha(b))(\gamma'),
\end{align*}
so $\pi_\alpha$ is $\Gamma$-equivariant. Borel measurability is immediate. \\

(iii) For $\gamma \in \Gamma$, let $h_\gamma(b) = \alpha(\gamma,b) - \psi_o(\gamma)$. 
Note that 
\begin{align*}
\mathscr{P}_m h_\gamma(\gamma') = \int_{B} (\alpha(\gamma,\gamma'.b) - \psi_o(\gamma)) \, dm(b) = \psi_o(\gamma \gamma') - \psi_o(\gamma') - \psi_o(\gamma),
\end{align*}
for all $\gamma, \gamma' \in \Gamma$. In particular, since $\mathscr{P}_m$ is an isometry, 
\[
\sup_{\gamma \in \Gamma} \|h_\gamma\|_\infty = \sup_{\gamma, \gamma' \in \Gamma} |\psi_o(\gamma \gamma') - \psi_o(\gamma') - \psi_o(\gamma)| = D(\psi_o) < \infty.
\]
In particular, for every $b \in B$, we have $|\alpha(\gamma,b) - \varphi_o(\gamma)| \leq D(\psi_o)$ for all $\gamma \in \Gamma$, so $\pi_\alpha(b)$ is indeed equivalent to $\psi_o$.
\end{proof}

To finish the proof of Lemma \ref{RightHarmonic} it now suffices to establish the following:
\begin{lemma} Let $\alpha$ as in Lemma \ref{CocycleAlphaRight} and define $\nu_{\psi_o} := (\pi_\alpha)_*m \in \mathrm{Prob}_p(\mathscr{F}_o(\Gamma))$.
\begin{enumerate}[(i)]
\item $\mathrm{supp}(\nu_{\psi_o}) \subset \Delta_{\psi_o}$, and hence $(\Delta_{\psi_o}, \nu_{\psi_o})$ is a $(\Gamma,p)$-boundary. \vspace{0.1cm}
\item If $\nu_{\psi_o} = \delta_\psi$ for some $\psi \in \mathscr{F}_o(\Gamma)$, then $\psi_o$ (hence $\phi_o$) is cohomologically trivial. 
\end{enumerate}
\end{lemma}
\begin{proof} (i) We consider the countable subset $\cA = \{ \alpha(\gamma,\cdot) \, : \, \gamma \in \Gamma\} \subset \mathscr{L}^\infty(B,m)$. By Lemma \ref{Lemma_Poisson} we may fix a $p^{\otimes \bN}$-conull subset $\Omega \subset \Gamma^{\bN}$
and a Borel map $Z_\infty : \Omega \ra B$ such that $(Z_\infty)_*p^{\bN_o} = m|_{\Omega}$ and
\[
f(Z_\infty(\omega)) = \lim_{n \ra \infty} \cP f(Z_n(\omega)) \quad \textrm{for all $\omega \in \Omega$ and $f \in \cA$}.
\]
We then deduce from Lemma \ref{CocycleAlphaRight} that for all $\omega \in \Omega$ and $\gamma \in \Gamma$ we have
\begin{equation}\label{alphaconv}
\alpha(\gamma,Z_\infty(\omega)) 
= \lim_{n \ra \infty} \int_{B} \alpha(\gamma,Z_n(\omega).b) \, dm(b)
= \lim_{n \ra \infty} Z_n(\omega).\psi_o(\gamma).
\end{equation}
We now fix $\psi \in \supp(\nu_{\psi_o})$. For every finite set  $F \subset \Gamma$ and $\eps > 0$ we set
 \[
U_{F,\eps} = \{ \psi' \in \mathscr{F}_o(\Gamma) \, : \, |\psi'(\gamma) - \psi(\gamma)| < \eps, \enskip \textrm{for all $\gamma \in F$} \}.
\]
 We then define $V_{F,\eps} = \pi_\alpha^{-1}(U_{F,\eps})$ and $W_{F,\eps} = Z_\infty^{-1}(V_{F,\eps})$. Since $U_{F, \eps}$ is an open neighbourhood of $\psi$ and $\psi \in \supp(\nu_{\psi_o})$ we then have
 \begin{equation}\label{WPositiveMeasure}
 p^{\otimes \bN}(W_{F,\eps}) = m(V_{F,\eps}) =  \nu_{\psi_o}(U_{F,\eps}) > 0.
 \end{equation}
 Note that explicitly we have
 \[
 W_{F,\eps} = Z_\infty^{-1}(V_{F,\eps}) = \{ \omega \in \Omega \mid |\alpha(\gamma,Z_\infty(\omega)) - \psi(\gamma) | < \eps, \enskip \textrm{for all $\gamma \in F$}\}.
 \]
 By \eqref{alphaconv} and \eqref{WPositiveMeasure} we now find for eery finite set $F \subset \Gamma$ and every $\eps>0$ an index $n = n(F, \eps) \in \bN$ and a subset $\Omega_n \subset \Omega$ such that
 \[
p^{\otimes \bN}(\Omega_n \cap W_{F,\eps}) > 0 \qand
|\alpha(\gamma,Z_\infty(\omega)) - Z_n(\omega).\psi_o(\gamma)| < \eps \; \textrm{for all $\gamma \in F$ and $\omega \in \Omega_n$}.
\]
For all $\omega \in \Omega_n \cap W_{F,\eps}$ and all $\gamma \in F$ we then have
\[
|Z_n(\omega).\psi_o(\gamma) - \psi(\gamma)|
\leq  |Z_n(\omega).\psi_o(\gamma) - \alpha(\gamma,Z_\infty(\omega)|
+ |\alpha(\gamma,Z_\infty(\omega)) - \psi(\gamma)| < 2\eps.
\]
We now fix an increasing exhaustion of $\Gamma$ by finite sets $(F_k)$, and apply the argument above to $F_k$ and $\eps_k = 1/k$ for every $k$. This produces a sequence $(n_k = n_k(F_k, 1/k))$ of positive integers and a corresponding sequence $\Omega_{n_k} \subset \Omega$ of Borel sets such that for all $k \in \bN$ we have
$p^{\otimes \bN}(\Omega_{n_k} \cap W_{F_k,1/k})> 0$ and 
\[
|Z_{n_k}(\omega).\psi_o(\gamma) - \psi(\gamma)| < 1/k \quad \textrm{for all $\omega \in \Omega_{n_k} \cap W_{F_k,1/k}$ and $\gamma \in F_k$}.
\]
Hence, if we pick $\omega_k \in \Omega_{n_k} \cap W_{F_k,1/k}$ and set $\gamma_k = Z_{n_k}(\omega_k)$, then
\[
|\gamma_k.\psi_o(\gamma) - \psi(\gamma)| < 1/k \quad \textrm{for all $\gamma \in F_k$}.
\]
Since $(F_k)$ exhausts $\Gamma$, we conclude that $\gamma_k.\psi_o \ra \psi$ as $k \ra \infty$, and thus $\psi \in \Delta_{\psi_o}$.

\item (ii) If $\nu_{\psi_o} = \delta_\psi$ for some $\psi \in \mathscr{F}_o(\Gamma)$, then $m(\pi_\alpha^{-1}(\psi)) = 1$. Thus for $m$-almost all $b$ we have $\pi_\alpha(b) = \psi$, and hence $\alpha(\gamma, b) = \psi(\gamma)$ for all $\gamma\in \Gamma$. By Lemma \ref{CocycleAlphaRight}.(iii) this implies that
 $\psi_o(\gamma\gamma') - \psi_o(\gamma') = \psi(\gamma)$. If we set $\gamma' = e$, then we obtain $\psi = \psi_o$. Consequently, $\psi_o$ is a homomorphism, hence cohomologically trivial. 
\end{proof}
This finishes the proof of Lemma \ref{RightHarmonic} and thereby of Theorem \ref{QuasiMain}.

\section{Stationary measures on hulls of quasimorphisms}\label{SecCore2}
\subsection{Reduction to the left-harmonic core}
Throughout this section let $(\Gamma, p)$ be a symmetric measured group and let $\phi_o: \Gamma \to \bR$ be a $p$-integrable quasimorphism. 
If we assume that $\varphi$ is cohomologically non-trivial, then it follows from Proposition \ref{QMTrivialities} that
\[
\mathrm{Prob}_\Gamma(\Delta_{\varphi_o}) = \emptyset, \quad \text{but} \quad \mathrm{Prob}_p(\Delta_{\varphi_o}) \neq \emptyset.
\]
This motivates a closer study of the space $\mathrm{Prob}_p(\Delta_{\varphi_o})$. If $\varphi_o$ happens to be left-$p$-harmonic, then we can describe this space directly using Theorem \ref{QuasiMain}:
\begin{corollary}\label{LeftHarmonicCase} If $\varphi_o$ is left-$p$-harmonic and cohomologically non-trivial, then $\nu_{\varphi_o}$ is the unique $p$-stationary measure on $\Delta_{\varphi_o}$ and $(\Delta_{\varphi_o}, \nu_{\varphi_o})$ is a non-trivial $(\Gamma, p)$-boundary.\qed
\end{corollary}
If $\varphi_o$ is not assume to be left-$p$-harmonic, then the space $\mathrm{Prob}_p(\Delta_{\varphi_o})$ is much harder to understand. Nevertheless we can get a partial understanding of $p$-stationary measures on $\Delta_{\varphi_o}$ by relating them to the unique $p$-stationary measure $\nu_{\varphi_o}$ on the left-$p$-harmonic core $\cC_p(\varphi_o)$ of $\varphi_o$ (cf.\ Theorem \ref{QuasiMain}). In fact we will show that for every (not necessarily $\Gamma$-ergodic) $p$-stationary probability measure $\nu$
on $\Delta_{\varphi_o}$, there is a $\nu$-almost everywhere defined Borel $\Gamma$-map on $\Delta_{\varphi_o}$ which maps into the core $\cC_p(\varphi_o)$ and takes $\nu$
to $\nu_{\varphi_o}$:
\begin{theorem}
\label{Theorem_Factor}
Let $\nu$ be a $p$-stationary probability measure on $\Delta_{\varphi_o}$. Then there
is a $\nu$-conull and $\Gamma$-invariant Borel set $\Delta_o \subset \Delta_{\varphi_o}$ and a Borel $\Gamma$-map $\pi : \Delta_o \ra \cC_p(\varphi_o)$ such that $\pi_*(\nu|_{\Delta_o}) = \nu_{\varphi_o}$.
\end{theorem}
\begin{remark}
We will in fact prove that the $\Gamma$-map $\pi$ can be everywhere defined on the hull $\Delta_{\varphi_o}$, but only as a \emph{universally measurable} $\Gamma$-map. Our construction of $\pi$ can be done in quite a straightforward manner if one is willing to assume the existence of so called Mokbodski means on $\bN$ (see Lemma \ref{Lemma_Mokobdoski} below). Unfortunately, these means cannot be constructed within 
ZFC, but requires us to adopt the Martin axiom. However, with a little bit more effort, the construction can be carried out within ZFC, and we refer to the reader to \cite[Proposition 2.3]{BH11} where a similar construction is made using Komlos' theorem (and thus entirely within ZFC).
\end{remark}
\subsection{Universally measurable sets and Mokobodski means}
Before we prove Theorem \ref{Theorem_Factor} we recall some basic definitions concerning universal measurability and Mokobodski means. 
\begin{remark}[Universally measurable sets]
Let $(X,\mathscr{B}_X)$ be a measurable space. If $\mu$ is a probability measure on $\mathscr{B}_X$, let $\mathscr{B}_X^\mu$ denote the $\mu$-completion of $\mathscr{B}_X$, i.e. the $\sigma$-algebra on $X$ generated by $\mathscr{B}_X$ and all subsets of $\mu$-null sets in $\mathscr{B}_X$. For instance, if $\mu = \delta_x$ for some $x \in X$, then $\mathscr{B}_X^{\mu}$ consists of all subsets of $X$. We define
\[
\mathscr{B}_X^* = \bigcap_{\mu \in \Prob(X)} \mathscr{B}_X^{\mu}.
\]
The elements in $\mathscr{B}_X^*$ are said to be \emph{universally measurable}. If $(Y, \mathscr{B}_Y)$ is another measurable space, then a map $f: X \to Y$ is called \emph{universally measurable} if it is $\mathscr{B}_X^*$-$\mathscr{B}_Y$-measurable.
\end{remark}
The following result is essentially due to Mokobodski \cite[Section 2]{M73} (see \cite[Proposition A.1]{BH11} for this particular version).
\begin{lemma}
\label{Lemma_Mokobdoski}
There is an invariant mean $\mathfrak{m}$ on $\bN$ with the property that whenever
$(X,\mathscr{B}_X)$ is a measurable space and $(f_n)$ is a uniformly bounded sequence
of $\mathscr{B}_X$-measurable functions on $X$, then the function
\[
f^*(x) = \int_{\bN} f_n(x) \, d\mathfrak{m}(n), \quad x \in X,
\]
is $\mathscr{B}_X^*$-measurable and moreover, the equality 
\[
\int_X f^*(x) \, d\mu(x) = \int_{\bN} \left( \int_X f_n(x) \, d\mu(x) \right) \, d\mathfrak{m}(n)
\]
holds for every Borel probability measure $\mu$ on $X$. 
\end{lemma}
Any invariant mean $\mathfrak{m}$ as in Lemma \ref{Lemma_Mokobdoski} is called a \emph{Mokobodski mean}.
\subsection{Proof of Theorem \ref{Theorem_Factor}}
The following lemma is a more precise version of Theorem \ref{Theorem_Factor}.
\begin{lemma}
In the situation of Theorem \ref{Theorem_Factor}, let $\mathfrak{m}$ be a Mokobodski mean. Gicen $\varphi \in \Delta_{\varphi_o}$ and $\gamma \in \Gamma$ we define
\begin{equation}\label{piMap}
\pi(\varphi)(\gamma) = \int_{\bN} \left( \sum_{\gamma \in \Gamma'}
p^{*n}(\gamma') (\varphi(\gamma'\gamma) - \varphi(\gamma')) \right) \, d\mathfrak{m}(n).
\end{equation}
Then the following hold:
\begin{itemize}
\item[(i)] \eqref{piMap} defines a $\Gamma$-equivariant and universally measurable map $\pi : \Delta_{\varphi_o} \ra \mathscr{F}_o(\Gamma)$. Furthermore,
\[
\|\pi(\varphi) - \varphi_o\|_\infty \leq 2D(\varphi_o), \quad \textrm{for all $\varphi \in \Delta_{\varphi_o}$},
\] 
and thus $\pi(\Delta_{\varphi_o}) \subset \varphi_o + [-2D(\varphi_o),2D(\varphi_o)]^{\Gamma}$. In particular, $\overline{\pi(\Delta_{\varphi_o})} \subset \mathscr{F}_o(\Gamma)$ is a separable metrizable space. \vspace{0.1cm}

\item[(ii)] For every $\varphi \in \Delta_{\varphi_o}$, the function $\gamma \mapsto \pi(\varphi)(\gamma)$ is a left-$p$-harmonic quasimorphism which is equivalent to $\varphi_o$. In particular, $\pi(\Delta_{\varphi_o})$ is contained in the left-$p$-harmonic core $\cC_p(\varphi_o)$. \vspace{0.1cm}
\item[(iii)] For every $p$-stationary probability measure $\nu$ on $\Delta_{\varphi_o}$,
there is a $\nu$-conull $\Gamma$-invariant Borel set $\Delta_o \subset \Delta_{\varphi_o}$ such that $\pi|_{\Delta_o} : \Delta_o \ra \mathscr{F}_o(\Gamma)$ is a Borel $\Gamma$-map and $\pi_*(\nu|_{\Delta_o}) = \varphi_{\varphi_o}$, where $\nu_{\varphi_o}$ denotes the unique $p$-stationary probability measure on $\cC_p(\varphi_o)$.
\end{itemize}
\end{lemma}

\begin{proof}
(i) For fixed $\gamma \in \Gamma$, define 
\[
\pi_n(\varphi)(\gamma) = \sum_{\gamma' \in \Gamma} p^{*n}(\gamma') (\varphi(\gamma'\gamma) - \varphi(\gamma')).
\]
Note that $\|\pi_n(\cdot)(\gamma)\|_\infty \leq 2D(\varphi_o) + |\varphi_o(\gamma)|$ for all $n$, and thus $(\pi_n(\cdot)(\gamma))$ is a uniformly bounded sequence of $\mathscr{B}_{\Delta_{\varphi_o}}$-measurable functions on $\Delta_{\varphi_o}$. 
Hence
\[
\pi(\varphi)(\gamma) = \int_{\bN} \pi_n(\varphi)(\gamma) \, d\mathfrak{m}(n), 
\]
is well-defined for all $(\varphi,\gamma) \in \Delta_{\varphi_o} \times \Gamma$ and $\varphi \mapsto \pi(\varphi)(\gamma)$ is a $\mathscr{B}_{\Delta_{\varphi_o}}^*$-measurable function on $\Delta_{\varphi_o}$ for every $\gamma \in \Gamma$. It readily follows that $\pi$ is universally measurable as a function from $\Delta_{\varphi_o}$ to $\mathscr{F}_o(\Gamma)$. To show that $\pi$ is $\Gamma$-equivariance, it suffices to check that $\pi_n$ is $\Gamma$-equivariant for every $n$. Note that
\begin{align*}
\pi_n(\gamma_o.\varphi)(\gamma) 
&=  \sum_{\gamma' \in \Gamma} p^{\ast n}(\gamma')(\gamma_o.\varphi(\gamma'\gamma)-\gamma_o.\varphi(\gamma'))\\[0.2cm]
&=  \sum_{\gamma' \in \Gamma} p^{\ast n}(\gamma')(\varphi(\gamma'\gamma\gamma_o) - \varphi(\gamma_o)-\varphi(\gamma'\gamma_o)+\varphi(\gamma_o)) \\[0.2cm]
&=  \sum_{\gamma' \in \Gamma} p^{\ast n}(\gamma')(\varphi(\gamma'\gamma\gamma_o) - \varphi(\gamma')-\varphi(\gamma'\gamma_o)+\varphi(\gamma')) \\[0.2cm]
&= \pi_n(\varphi)(\gamma\gamma_o)-\pi_n(\varphi)(\gamma_o)  =  (\gamma_o.\pi_n(\phi))(\gamma),
\end{align*}
for all $\varphi \in \Delta_{\varphi_o}$ and $\gamma_o, \gamma \in \Gamma$, and thus $\pi_n$ is $\Gamma$-equivariant. (ii) 
To see that $\gamma \mapsto \pi(\varphi)(\gamma)$ is left-$p$-harmonic, we compute
\begin{align*}
(p \ast \pi_n(\varphi))(\gamma) 
&= \sum_{\gamma_o \in \Gamma} p(\gamma_o)\pi_n(\varphi)(\gamma_o^{-1}\gamma) = \sum_{\gamma_o \in \Gamma} \sum_{\gamma' \in \Gamma} p(\gamma_o) p^{\ast n}(\gamma')(\varphi(\gamma'\gamma_o^{-1}\gamma)-\varphi(\gamma'))\\[0.2cm]
&= \sum_{\gamma'' \in \Gamma} p^{\ast(n+1)}(\gamma') \varphi(\gamma''\gamma) - \sum_{\gamma' \in \Gamma} p^{\ast n}(\gamma') \varphi(\gamma') \\[0.2cm]
&= \pi_{n+1}(\varphi)(\gamma) + r_{n+1}(\varphi) - r_n(\varphi),
\end{align*}
where $r_n(\varphi) := \sum_{\gamma' \in \Gamma} p^{\ast n}(\gamma') \phi(\gamma')$. We claim that $(r_n)$ is a uniformly bounded sequence of functions on $\Delta_{\varphi_o}$. Indeed, let $\eta_o$ denote the homogeneous quasi-morphism associated with $\varphi_o$. Then $\|\varphi - \eta_o\|_\infty < \infty$ for all $\varphi \in \Delta_{\varphi_o}$ and 
\[
\sum_{\gamma' \in \Gamma} p^{\ast n}(\gamma') \eta_o(\gamma') = 0,
\]
since $p$ is symmetric and $\eta_o$ is anti-symmetric. Hence, for all $n$,
\[
|r_n(\varphi)| = \left| 
\sum_{\gamma' \in \Gamma} p^{\ast n}(\gamma') (\varphi -\eta_o)(\gamma') \right|
\leq D(\varphi_o) + \|\varphi_o - \eta_o\|_\infty, \quad \textrm{for all $\varphi \in \Delta_{\varphi_o}$}.
\]
Since $\mathfrak{m}$ is invariant and $(r_n(\varphi))$ is bounded, we deduce that
\begin{align*}
\int_{\bN} (p * \pi_n(\varphi))(\gamma) \, d\mathfrak{m}(n)
&= 
\int_{\bN} (\pi_{n+1}(\varphi))(\gamma) \, d\mathfrak{m}(n) + \int_{\bN} (r_{n+1}(\varphi) - r_n(\varphi)) \, d\mathfrak{m}(n) \\[0.2cm]
&= \pi(\varphi)(\gamma), \quad \textrm{for all $\varphi \in \Delta_{\varphi_o}$ and $\gamma \in \Gamma$}.
\end{align*}
It remains to show that
\begin{equation}
\label{remains}
\int_{\bN} (p * \pi_n(\varphi))(\gamma) \, d\mathfrak{m}(n) = (p * \pi(\varphi))(\gamma), \quad \textrm{for all $\gamma \in \Gamma$}.
\end{equation}
To do this, fix $\varphi \in \Delta_{\varphi_o}$, and note that $u_n(\gamma) = \pi_n(\varphi)(\gamma) - \varphi(\gamma)$ satisfies $\|u_n\|_\infty \leq D(\varphi_o)$
for all $n$, and
\[
u_\infty(\gamma) = \int_{\bN} u_n(\gamma) \, d\mathfrak{m}(n) = \pi(\varphi)(\gamma) - \varphi(\gamma), \quad \gamma \in \Gamma
\]
In particular,
\[
(p * u_\infty)(\gamma) = (p * \pi(\varphi))(\gamma) - (p*\varphi)(\gamma), \quad \textrm{for all $\gamma \in \Gamma$}.
\]
Furthermore, by the second property of $\mathfrak{m}$ in Lemma \ref{Lemma_Mokobdoski}, applied with $X = \Gamma$ and $\mu = p * \delta_{\gamma}$ for fixed $\gamma \in \Gamma$, we see that
\[
(p * u_\infty)(\gamma) = \int_{\bN} (p * u_n)(\gamma) \, d\mathfrak{m}(n)
= \int_{\bN} (p * \pi_n(\varphi))(\gamma) \, d\mathfrak{m}(n) - (p*\varphi)(\gamma).
\]
Comparing the last two expression proves \eqref{remains}. 
(iii) Let $\nu$ be a probability measure on $\Delta_{\varphi_o}$. 
Since $\pi$ is a universally measurable map from $\Delta_{\varphi_o}$ into a separable metric space by (i), \cite[Lemma 1.2]{Crauel} tells us that we can find a $\nu$-conull Borel set $\Delta' \subset \Delta_{\varphi_o}$ such that $\pi|_{\Delta'}$ is Borel. If $\nu$ in addition is $p$-stationary, then $\nu$ is $\Gamma$-quasi-invariant, and thus $\Delta_o = \bigcap_{\gamma \in \Gamma} \gamma.\Delta'$ is a $\Gamma$-invariant $\nu$-conull Borel set, contained in $\Delta'$, so clearly the restriction of $\pi$ to $\Delta_o$ is again Borel.
\end{proof}

\section{Hulls with many stationary measure}\label{SecMany}
Let $(\Gamma, p)$ be a symmetric measured group and let $\phi_o: \Gamma \to \bR$ be a $p$-integrable quasimorphism.
We have seen in Corollary \ref{LeftHarmonicCase} that if $\phi_o$ is left-$p$-harmonic, then there is a unique $p$-stationary measure on $\Delta_{\varphi_o}$. If $\phi_o$ is not assumed to be left-$p$-harmonic then this uniqueness statement is lost; in fact we are going to show:
\begin{theorem}\label{ManyStationaryMeasures} 
Every quasimorphism $\phi_o: \Gamma \to \R$ is equivalent to a quasimorphism $\phi_1: \Gamma \to \bZ$ such that $|\mathrm{Prob}^{\mathrm{erg}}_p(\Delta_{\varphi_1})| = 2^{\aleph_o}$. If $\Gamma$ surjects onto $\bZ$, then $\phi_1: \Gamma \to \bZ$ can be chosen to be antisymmetric.
\end{theorem}
Note that the elements of $\Delta_{\varphi_1}$ are still heavily restricted by Theorem \ref{Theorem_Factor}.
\subsection{Reduction to the case of bounded functions}
We start the proof of Theorem \ref{ManyStationaryMeasures} with some easy reductions.
Since by Remark \ref{QMBasics} every real-valued quasimorphism $\phi_o$ is equivalent to an antisymmetric quasimorphism which takes values in $3\bZ$, we may as well assume that this is the case for $\phi_o$. We are then going to define
\begin{equation}
\phi_1 := \phi_o + s_o,
\end{equation}
where $s_o \in \mathscr{F}_o(\Gamma)$ is a carefully chosen antisymmetric function with $s_o(\Gamma) \subset \{-1,0,1\}$. The proof of Theorem \ref{ManyStationaryMeasures} will be based on the following observation:
\begin{lemma}\label{UncountableMain}
Let $\varphi_o: \Gamma \to 3\bZ$ be an (antisymmetric) quasimorphism and let $\varphi_1 := \varphi_o + s_o$, where
$s_o \in \mathscr{F}_o(\Gamma)$ be (antisymmetric) with $s_o(\Gamma) \subset \{-1,0,1\}$. Then $\varphi_1: \Gamma \to \Z$ is an (antisymmetric) quasimorphism and there is a $\Gamma$-equivariant continuous 
surjection $\beta: \Delta_{\varphi_1} \ra \Delta_{s_o}$.
\end{lemma}
\begin{proof}
Let $\Delta_{(\varphi_o,s_o)}$ denote the joint hull of $(\varphi_o,s_o)$, i.e. 
\[
\Delta_{(\varphi_o,s_o)} = \overline{\Gamma.(\varphi_o,s_o)} \subset \Delta_{\varphi_o} \times \Delta_{s_o}.
\]
We claim that the map 
\[
\alpha : \Delta_{(\varphi_o,s_o)} \ra \Delta_{\varphi_o+s_o}, \enskip (\varphi,s) \mapsto \varphi+s
\]
is a $\Gamma$-equivariant homeomorphism. Indeed, continuity, surjectivity and $\Gamma$-equivariance of $\alpha$ are easy to check, and since $\Delta_{(\varphi_o,s_o)}$ is compact, it suffices to check that the map is injective. To see this, pick two pairs $(\varphi_1,s_1)$ and $(\varphi_2,s_2)$ in $\Delta_{(\varphi_o,s_o)}$ such that 
$\varphi_1 + s_1 = \varphi_2 + s_2$. Then, since both $\varphi_1(\Gamma)$ and $\varphi_2(\Gamma)$ are contained in $3\bZ$ and $s_1$ and $s_2$ take values in $\{-1,0,1\}$, we get
\[
\varphi_1 - \varphi_2 = s_2 - s_1 \in 3\bZ \cap \{-2,-1,0,1,2\} = \{0\},
\]
and thus $\varphi_1 = \varphi_2$ and $s_1 = s_2$. To construct the map $\beta$, we compose the inverse of $\alpha$ with the projection $\Delta_{(\varphi_o,s_o)} \ra \Delta_{s_o}$, $(\varphi, s) \mapsto \varphi$.
\end{proof}
We deduce in particular, that
\begin{equation}
|\Prob_p^{\mathrm{erg}}(\Delta_{\varphi_1})| \geq |\Prob_p^{\mathrm{erg}}(\Delta_{s_o})|.
\end{equation}
We have thus reduced the proof of Theorem \ref{ManyStationaryMeasures} to the case where $\phi_o = s_o$ takes values in $\{-1,0,1\}$.
\subsection{The unconstrained case}
To finish the proof of Theorem \ref{ManyStationaryMeasures} we need to construct an antisymmetric function $s_o: \Gamma\to \{-1,0,1\}$ with $s_o(e) = 0$ such that 
$|\mathrm{Prob}^{\mathrm{erg}}_p(\Delta_{s_o})| = 2^{\aleph_o}$. If we do not insist of antisymmetry of $\phi_1$, then constructing such a map $s_o$ is easy:
\begin{construction}\label{etaForS0} Consider the $\Gamma$-space $2^\Gamma_* := 2^\Gamma \setminus\{\emptyset\}$ with $\Gamma$-action given by $\gamma. A := A \gamma^{-1}$.
We have a continuous $\Gamma$-equivariant map
\[
\eta: 2^\Gamma \to \mathscr{F}_o(\Gamma), \quad \eta(A)(\gamma) := \chi_A(\gamma) - \chi_A(e),
\]
which is injective on $2^\Gamma_*$ and satisfies $\eta(A)(\Gamma) \in \{-1,0,1\}$ for every $A \in 2^\Gamma_*$. If we choose $A_o$ in such a way that the $\Gamma$-orbit of $A_o$ is dense in $2^\Gamma_*$ and set $s_o := \eta(A)$, then $\eta$ restricts to an equivariant embedding $2^\Gamma_* \hookrightarrow \Delta_{s_o}$. This readily implies that $ |\Prob_p^{\mathrm{erg}}(\Delta_{s_o})| = 2^{\aleph_o}$.
\end{construction}
This proves the first part of Theorem \ref{ManyStationaryMeasures}.
\subsection{The antisymmetric case}
To prove the second part of Theorem \ref{ManyStationaryMeasures} we need to find an \emph{antisymmetric} function $s_o: \Gamma\to \{-1,0,1\}$ with $s_o(e) = 0$ such that 
$|\mathrm{Prob}^{\mathrm{erg}}_p(\Delta_{s_o})| = 2^{\aleph_o}$, wherewe assume the existence of a surjection $\delta : \Gamma \to \bZ$. We may then assume without loss of generality that $\Gamma=\bZ$. Indeed, if $s_o \in \mathscr{F}_o(\bZ)$ has the desired properties, then so has $\delta^*s_o \in \mathscr{F}_o(\Gamma)$. We thus consider the antisymmetrization of the map $\eta: 2^\bZ_\ast \to \mathscr{F}_o(\bZ)$ from Construction \ref{etaForS0} as given by
\[
\xi: 2^{\bZ}_\ast \to \mathscr{F}_o(\bZ), \quad \xi(A) := \eta(A) - \eta(-A) = \chi_A-\chi_{-A}.
\]
By construction, this map is continuous, and $\xi(A)$ is antisymmetric and takes values in $\{-1,0,1\}$ for all $A \in 2^{\bZ}_\ast$.
\begin{definition} We say that $A \subset \bZ$ is \emph{left-generic} (respectively \emph{right-vanishing}) if
\[
\overline{\{ n.A \, : \, n \geq 0 \}} = 2^{\bZ} \quad (\text{respectively} \lim_{n \ra -\infty} n.A = \emptyset).
\]
\end{definition}
\begin{lemma} If $A$ is left-generic and right-vanishing and $s_o = \xi(A)$, then \[
\Delta_{s_o} = \bZ.s_o \cup \eta(2_*^{\bZ}) \cup (-\eta(2_*^{\bZ})), \quad \text{and hence}\quad |\Prob_p^{\mathrm{erg}}(\Delta_{s_o})| = 2^{\aleph_o}.
\]
\end{lemma}
\begin{proof}
Note that for all $k, n \in \bZ$ we have
\begin{align*}
(k.s_o)(n) 
&=
s_o(n+k) - s_o(k) = \chi_A(n+k) - \chi_A(-n-k) - \chi_A(k) + \chi_A(-k)\\
&= 
(\chi_{A-k}(n) - \chi_{A-k}(0)) - (\chi_{A+k}(-n) - \chi_{A+k}(0)),
\end{align*}
which we can rewrite as
\begin{equation}\label{kso}
(k.s_o)(n)  = \eta(k.A)(n) - \eta((-k).A)(-n) \quad \text{for all }k,n \in \bZ.
\end{equation}
To show that $\eta(2^{\bZ}_\ast) \subset \Delta_{s_o}$, let $B \in 2_\ast^{\bZ}$. Since $A$ is left-generic and right-varnishing, we find a sequence $k_j \to \infty$ such that $k_j.A \ra B$ and $(-k_j).A \ra \emptyset$. We thus conclude from \eqref{kso} that
\begin{align*}
\eta(B)(n) =  \eta(B)(n) - \eta(\emptyset)(-n) = \lim_{j \to \infty}  \eta(k_j.A)(n) - \eta((-k_j).A)(-n) = \lim_{j \to \infty} (k_j.s_o)(n)
\end{align*}
for all $n \in \bZ$, and hence $\eta(B) = \lim_{j \to \infty} k_j.s_o \in \Delta_{s_o}$. Since $B \in 2_\ast^\bZ$ was arbitrary, this shows that $\eta(2^{\bZ}_\ast) \subset \Delta_{s_o}$. Replacing the sequence $(k_j)$ with $(-k_j)$, we see that 
\[
-\eta(-B)(n) =  \eta(\emptyset)(n) - \eta(B)(-n) = \lim_{j \to \infty}((-k_j).s_o)(n),
\]
for all $n \in \bZ$, and hence $-\eta(-B) =  \lim_{j \to \infty} (-k_j).s_o \in \Delta_{s_o}$. Consequently, $-\eta(2_*^{\bZ}) \subset \Delta_{s_o}$,
and thus
\begin{equation}\label{OrbitClosureso}
\bZ.\xi_A \cup \eta(2_\ast^{\bZ}) \cup (-\eta(2_\ast^{\bZ})) \subset \Delta_{\xi_A}.
\end{equation}
Finally, if $(k_j)$ is any unbounded sequence, then upon passing to a subsequence, we may assume that $(k_j)$ tends to either $+\infty$ or $-\infty$ and that $(\pm k_j).A$ converge. There thus exists $B \in 2^{\bZ}$ so that either
\[
(k_j.A \ra B \text{ and } (-k_j).A \ra \emptyset) \qor (k_j.A \ra \emptyset  \text{ and } (-k_j).A \ra B).
\]
In the first case, $k_j.s_o \ra \eta(B)$, while in the second case, $k_j.s_o \ra -\eta(-B)$. This shows that equality holds in \eqref{OrbitClosureso} and finishes the proof.
\end{proof}
To conclude the proof of Theorem \ref{ManyStationaryMeasures} we thus have to construct a left-generic and right-vanishing subset of $\bZ$. 
\begin{construction}
We consider the one-sided Bernoulli shift 
\[
T : \{0,1\}^{\bN} \ra \{0,1\}^{\bN}, \enskip (T\omega)_k = \omega_{k+1}, \enskip k \geq 0.
\]
The product measure $\mu = \left(\frac{1}{2}\delta_o + \frac{1}{2}\delta_1\right)^{\otimes \bN}$ is $T$-invariant and ergodic, hence $\mu$-almost every point $x$ in $\{0,1\}^{\bN}$ has a dense $T$-orbit. For any such $x$ we set
\[
A(x) : = \{ k \in\ \bZ \, : \, k \geq 0, \enskip x_k = 1\}.
\]
We then observe that for all $n \geq 0$ we have
\begin{align*}
A(T^n x) &= \{ k \in \bZ \, : \, k \geq 0, \enskip x_{k+n} = 1\} = (A(x) - n) \cap \bN.
\end{align*}
\end{construction}
\begin{lemma} If $x \in \{0,1\}^{\bN}$ has a dense $T$-orbit, then $A(x)$ is left-generic and right-vanishing.
\end{lemma}
\begin{proof} Since the $T$-orbit of $x$ is dense, we see that for every $y \in \{0,1\}^{\bN}$ and for every finite set $F \subset \bN$, there exists $n = n(y,F)$ such 
that
\begin{equation}
\label{AF}
(A(x)-n) \cap F = A(y) \cap F.
\end{equation}
We claim that this implies that $A$ is left-generic. Indeed, fix $B \in 2^{\bZ}$ and a finite set $F \subset \bZ$. Choose a large positive integers $m$ such that $F + m \subset \bN$. Define $y \in \{0,1\}^{\bN}$ by
\[
y_k = 1 \iff k \in (B+m) \cap \bN.
\]
Note that $A(y) = (B+m) \cap \bN$. By \eqref{AF}, we can find $n$ such that
\[
(A(x) - n) \cap (F+m) = A(y) \cap (F+m) = (B+m) \cap (F+m) = (B \cap F) + m,
\]
and thus
\[
(A(x) - (n+m)) \cap F = B \cap F.
\]
In other words, for every finite set $F \subset \bZ$, we can find $k \geq 0$ such that
\[
(A(x)-k) \cap F = B \cap F.
\]
Since $B$ is arbitrary, this shows that $\overline{\{ k.A \, : \, k \geq 0\}} = 2^{\bZ}$, and thus $A(x)$ is left-generic. Furthermore, since $A(x) \subset \bN$, we have
\[
n.A(x) = A(x)-n \longrightarrow \emptyset \quad \textrm{as $n \ra -\infty$},
\]
so $A(x)$ is right-vanishing as well. 
\end{proof}
This finishes the proof of Theorem \ref{ManyStationaryMeasures}.

\section{Hulls of twisted sets}\label{SecTwisted}
\subsection{Twisted sets and uniform approximate lattices}
Throughout this section let $(\Gamma,p)$ be a symmetric measured group. We recall from the introduction that if $\varphi_o: \Gamma \to \R$ is a quasimorphism and $P_o \subset \bR$ is a closed subset, then the $\varphi_o$-twist of $P_o$ is defined as the subset
\[
P_o(\varphi_o) 
= 
\{ (\gamma, t) \in \Gamma \times \R
\mid
\varphi_o(\gamma)+t \in P_o \} \subset \Gamma \times \bR.
\]
Note that $P_o(\varphi_o)$ is always a \emph{closed} subset of $\Gamma \times \bR$. Moreover, if $\pr_\Gamma$ and $\pr_\bR$ denote the projections onto the two factors of $\Gamma \times \bR$, then $\pr_\Gamma(P_o(\varphi_o)) = \Gamma$ and $\pr_\bR(P_o(\varphi_o)) = P_o - \varphi_o(\Gamma)$. The twist construction can be used to produce interesting examples of approximate subgroups of $\Gamma \times \bR$; we recall that a subset $\Lambda \subset \Gamma \times \bR$ is called an approximate subgroup if it is symmetric, contains the identity and satisfies $\Lambda\Lambda \subset \Lambda F$ for some finite subsets $F\subset \Gamma \times \bR$.
\begin{proposition}\label{ApproxGrp} If $P_o \subset \R$ is an approximate subgroup and $\varphi_o$ is an antisymmetric quasimorphism with finite defect set, then $P_o(\varphi_o)$ is an approximate subgroup of $\Gamma \times \bR$. 
\end{proposition}
\begin{proof} Since $0 \in P_o$ and $\varphi_o(e) = 0$ we have $(e, 0) \in P_o(\varphi_o)$, and since $P_o \subset \bR$ is symmetric and $\varphi_o$ is anti-symmetric, the subset $P_o(\varphi_o) \subset \Gamma \times \bR$ is symmetric. If we denote by $\mathscr{D}(\varphi_o)$ the finite defect set of $\varphi_o$ and choose $F \subset \bR$ finite such that $P_o + P_o \subset P_o + F$, then
\begin{align*}
P_o(\varphi_o) P_o(\varphi_o)
&\subseteq \{ (\gamma_1 \gamma_2,t_1+t_2) \mid t_1 + t_2 + 
\varphi_o(\gamma_1) + \varphi_o(\gamma_2) \in P_o + P_o \} \\[0.2cm]
&\subseteq \{ (\gamma_1 \gamma_2,t_1 + t_2) \mid t_1 + t_2 + \varphi(\gamma_1 \gamma_2) \in P_o + F + \mathscr{D}(\varphi_o) \} \\[0.2cm]
&=
P_o(\varphi_o)(\{e\} \times  (F+\mathscr{D}(\varphi_o))),
\end{align*}
and thus $P_o(\varphi_o)$ is an approximate subgroup of $\Gamma \times \bR$.
\end{proof}
\begin{remark}[Twisted Delone sets]\label{DeloneSets}
If $G$ is a locally compact group, then a subset $\Lambda \subset G$ is called \emph{$Q$-relatively dense} for some compact set $Q \subset G$, if $\Lambda Q = G$ and \emph{$U$-uniformly discrete} for some identity neighbourhood $U$ if $\Lambda\Lambda^{-1} \cap U = \{e\}$. It is called a \emph{Delone set} if it is $Q$-relatively dense and $U$-uniformly discrete for some compact set $Q \subset G$ and some identity neighbourhood $U$. An approximate subgroup of $G$, which is also a Delone set is called a \emph{uniform approximate lattice}.

One readily checks that if $P_o \subset \bR$ is $U$-uniformly discrete or $Q$-relatively dense, then for every normalized quasimorphism $\phi_o$ the twisted set $P_o(\phi_o)$ is $\{e\} \times U$-uniformly discrete and $\{e\} \times Q$-relatively dense in $\Gamma \times \bR$. In particular, twists of Delone sets in $\bR$ are Delone sets in $\Gamma \times \bR$. If $\Gamma$ is a uniform lattice in some locally compact group $G$, they are thus also Delone sets in $G \times \bR$. 
\end{remark}
Since $\bZ$-valued quasimorphisms have a finite defect set, Proposition \ref{ApproxGrp} and Remark \ref{DeloneSets} imply the following corollary.
\begin{corollary}\label{GetUAL}
If $P_o$ is a uniform approximate lattice in $\bR$, $\Gamma$ is a uniform lattice in a locally compact group $G$ and $\varphi_o: \Gamma \to \bZ$ is an antisymmetric quasimorphism,
then $P_o(\varphi_o)$ is a uniform approximate lattice in $G \times \bR$.
\end{corollary}
For example, this construction allows us to construct uniform approximate lattices in $\mathrm{SL_2}(\R) \times \R$ by starting from a counting quasimorphism $\varphi_o: \Gamma_g \to \bZ$ on a surface group $\Gamma_g$. 

The remainder of this section is devoted to the study of hulls of twisted sets. We will now realize these as twisted skew products over hulls of quasimorphisms.

\subsection{Hulls of twisted sets as skew products}
Given a locally compact second countable group $G$ we denote by $\mathscr C(G)$ the space of closed subsets of $G$ equipped with the Chabauty-Fell topology and $G$-action given by $g.P := Pg^{-1}$. Then $\mathscr C(G)$ is a compact metrizable space, and a sequence $P_n$ in $\mathscr C(G)$ converges to some $P \in \mathscr C(G)$ if and only if the following two conditions hold:
\begin{enumerate}[(CF1)]
\item If $p_{n_k} \in P_{n_k}$ for some subsequence $(n_k)$ and $p_{n_k} \to p$ in $G$, then $p \in P$.
\item For every $p \in P$, there exist elements $p_n \in P_n$ such that $p_n \ra p$ in $G$.
\end{enumerate}
Given $P_o \in \mathscr{C}(G)$ we denote by $\Omega_{P_o}$ the orbit closure of $P_o$ in $\mathscr C(G)$ and refer to $\Omega_{P_o}$ as the \emph{hull} of $P_o$. 
We recall that if $\Delta$ is a compact space with a $\Gamma$-action by homeomorphisms, then a continuous map $c: \Gamma \times \Delta \to \bR$ is called a \emph{normalized cocycle} if $c(\gamma_1\gamma_2, y)  = c(\gamma_1,\gamma_2.y) + c(\gamma_2,y) \text{ and } c(e, y) = 0 \text{ for all } y \in\Delta \text{ and }\gamma_1, \gamma_2 \in \Gamma$.
\begin{definition} If $c: \Gamma \times \Delta \to \bR$ be a normalized  continuous cocycle and $Z$ is a compact $\R$-space, then the  \emph{$c$-twisted action} of $\Gamma \times \R$ on $\Delta \times Z$ is the continuous action given by
\[
(\gamma,t).(y,z) := (\gamma .y, (c(\gamma,y)+t).z) \quad (\gamma \in \Gamma, t \in \R, y \in \Delta, z \in Z).
\]
The associated $(\Gamma\times \R)$-space is denoted $\Delta \times_c Z$ and referred to as the \emph{skew product} of $\Delta$ and $Z$ over $c$.
\end{definition}
\begin{remark}\label{BaseFactor}
Note that the projection $\pi_1: \Delta \times_c Z \to \Delta$ onto the first factor is a $\Gamma$-equivariant continuous map. In particular, $\Delta$ is a continuous $\Gamma$-factor of the skew product $\Delta \times_c Z$. 
\end{remark}
We now investigate conditions on a set $\widetilde{P} \subset \Gamma \times \bR$ which guarantee that the hull of $\widetilde{P}$ splits as a skew product. Our starting point is the following observation:
\begin{proposition}\label{FactorMapCts} Let $c: \Gamma \times \Delta \to \bR$ be a continuous cocycle and let $P_o \subset \bR$ be closed. Then 
there is a  $(\Gamma \times \R)$-equivariant continuous map
\[\pi: \Delta \times_c \Omega_{P_o} \to \mathscr C(\Gamma \times \R), \; (y, P) \mapsto \left\{ (\gamma, t) \in \Gamma \times \bR \mid c(\gamma,y) + t \in P \right\}.\]
\end{proposition}
\begin{proof} 
Equivariance is straight-forward to check. To prove continuity, let $(y_n,P_n)$ be a sequence in $\Delta \times_c \Omega_{P_o}$ which converges to some $(y,P)$ in $\Delta \times_c \Omega_{P_o}$. Since $\mathscr{C}(\Gamma \times \bR)$ is compact, we can find $\widetilde{P} \in \mathscr{C}(\Gamma \times \bR)$ and a sub-sequence $(n_k)$ such that
\[
\pi(y_{n_k},P_{n_k}) \ra \widetilde{P} \quad \textrm{in $\mathscr{C}(\Gamma \times \R)$}.
\]
To prove continuity we need to show that $\widetilde{P} = \pi(y,P)$. To do this, pick $(\gamma, t) \in \widetilde{P}$. Then, by (CF2) we can find $(\gamma_{k},t_k) \in \pi(y_{n_k},P_{n_k})$ such
that $(\gamma_k,t_k) \ra (\gamma, t)$. Note that by definition
\[
p_{n_k} := c(\gamma_k,y_{n_k}) + t_k \in P_{n_k}, \quad \textrm{for all $k$}.
\]
Since $c$ is continuous, we conclude by (CF1) that $c(\gamma, t) + a \in P$, and thus $(\gamma, t) \in \pi(y,P)$. Since $(\gamma, t) \in \widetilde{P}$ was chosen arbitrarily, this shows $\widetilde{P} \subseteq \pi(y,P)$. 

\item Conversely, suppose $(\gamma,t) \in \pi(y,P)$ and define 
\[(y', P') := (\gamma,t).(y,P) \qand (y_n',P_n') := (\gamma,t).(y_n,P_n). 
\] 
By equivariance of $\pi$ we then have $(e, 0) \in \pi(y', P')$ and hence $0 = c(e, y') + 0 \in P'$.
Moreover, since $P_n \to P$ we have $P_n' \to P'$, and hence by (CF2) we find a sequence $t_n' \in P_n'$ such that $t_n' \to 0$. Then for all $n \in \bN$ we have, by equivariance of $\pi$,
\begin{align*}
c(e, y_{n}') + t_{n}' = t_{n}' \in P_{n}' &\; \implies (e, t_{n}') \in \pi(y_n', P_{n}') = \pi(y_n, P_n)(\gamma, t)^{-1}\\
&\; \implies (\gamma, t+t_n') \in \pi(y_n, P_n).
\end{align*}
Since $t_n' \to 0$ it then follows from (CF1), that
\[
(\gamma, t) = \lim_{k \to \infty} (\gamma, t+t_{n_k}') \in \lim_{k \to \infty} \pi(y_{n_k}, P_{n_k}) =  \widetilde{P}.
\]
Since this holds for all $(\gamma,t) \in \pi(y,P)$ we have $ \pi(y,P) \subset \widetilde{P}$ and hence $ \pi(y,P) = \widetilde{P}$.
\end{proof}
We will need the following property of the map $\pi$.
\begin{lemma}\label{Fiberspi} If the $\R$-action on $\Omega_{P_o}$ is free, then for all $(y,P) \in \Delta \times \Omega_{P_o}$ we have
\[
\pi^{-1}(\pi(y,P)) = \{ (y',P) \in \Delta \times \Omega_{P_o} \mid c(\gamma,y) = c(\gamma,y') \enskip \textrm{for all $\gamma \in \Gamma$} \}.
\]
\end{lemma}
\begin{proof} For $(y,P), (y', P') \in \Delta \times \Omega_{P_o}$ we have the equivalences
\begin{align*}
\pi(y,P) = \pi(y',P') &\iff  (c(\gamma,y) + a \in P \iff c(\gamma, y') + a \in P') \text{ for all $\gamma \in \Gamma, a \in A$}\\
&\iff P - c(h,y) = P' - c(h,y') \textrm{ for all $\gamma \in \Gamma$}.
\end{align*}
If we plug in $h=e$ into the latter statement, then we see that these statements imply that $P = P'$. Since the action $A$-action is free we deduce that
\begin{align*}
\pi(y,P) = \pi(y',P') & \Leftrightarrow \; (P = P') \text{ and } (P = P - (c(\gamma,y')-c(\gamma,y)) \text{ for all $\gamma \in \Gamma$})\\
& \Leftrightarrow \; (P = P') \text{ and } (c(h,y) = c(h,y') \text{ for all $\gamma\in \Gamma$}). \qedhere
\end{align*}
\end{proof}
This allows us to show that the hulls of certain subsets of $\Gamma \times \R$ are skew products:
\begin{corollary}\label{Critb} 
Let $\Delta$ be a compact metrizable space on which a countable group $\Gamma$ acts by homeomorphisms, let $c: \Gamma \times \Delta \to \bR$ be a normalized continuous cocycle and let $P_o \subset \bR$ be a closed subset. Assume that
\begin{enumerate}[(1)]
\item the $\R$-action on $\Omega_{P_o}$ is free;
\item $c$ separates points in $\Delta$, i.e.\ if for all $y_1, y_2 \in \Delta$ there exists $\gamma \in \Gamma$ such that $c(\gamma ,y_1) \neq c(\gamma, y_2)$.
\item some $y_o \in \Delta$ has a dense $\Gamma$-orbit.
\end{enumerate}
Then for any $y_o$ as in (3) the map $\pi$ from Proposition \ref{FactorMapCts} restricts to a $(\Gamma \times \R)$-equivariant homeomorphism $\Delta \times_c \Omega_{P_o} \to  \Omega_{\pi(y_o,P_o)}$. In particular, $\Delta$ is a continuous factor of $\Omega_{\pi(y_o,P_o)}$ and if $\mathrm{Prob}_\Gamma(\Delta) = \emptyset$, then $\mathrm{Prob}_{\Gamma \times \R}(\Omega_{\pi(y_o,P_o)}) = \emptyset$.
\end{corollary}
\begin{proof} By Lemma \ref{Fiberspi} and since $c$ separate points the non-empty fibers of $\pi$ are singletons; since $Y \times_c \Omega_{P_o}$ is compact and $\pi$ is continuous, this implies that $\pi$ is a homeomorphism onto its image. Since $(y, P_o)$ is dense in $Y \times_c \Omega_{P_o}$, then $\pi(y_o, P_o)$ is dense in this image by equivariance and continuity of $\pi$. This proves the first statement, and the other two statements readily follows from Remark \ref{BaseFactor} and Lemma \ref{NoInvMeasure}.
\end{proof}
Corollary \ref{Critb} applies in particular to hulls of certain twisted sets. We will work in the following setting:
\begin{setting}\label{MainSetting} Throughout this subsection we use the following notation:
\begin{enumerate}[(1)]
\item $\Gamma$ is a countable discrete group;
\item $\varphi_o: \Gamma \to \bR^d$ is a quasimorphism with hull $\Delta_{\varphi_o}$ denotes its hull;
\item $c: \Gamma \times \Delta_{\varphi_o} \to \bR$ denotes the canonical cocycle given by $c(\gamma, \varphi) = \varphi(\gamma)$;
\item $P_o \subset \bR$ is a closed subset such that the $\bR$-action on $\Omega_{P_o}$ is free;
\item  $\widetilde{p}$ is a probability measure on $\Gamma \times \R$ which is absolutely continuous with respect to the Haar measure on $\Gamma \times \R$ and whose support generates $\Gamma \times \R$ as a semigroup. We denote by $p$ the push-forward of $\widetilde{p}$ to $\Gamma$ and assume that $\varphi_o$ is $p$-integrable.
\end{enumerate}
\end{setting}
\begin{proposition}[Skew product realization]\label{SkewRealization} In the situation of Setting \ref{MainSetting}, consider the twisted set $P_o(\varphi_o) \subset \Gamma \times \bR$. 
Then the map
\begin{equation}\label{ThePi}\pi_{P_o, \varphi_o}: \Delta_{\varphi_o} \times_c \Omega_{P_o} \to \Omega_{P_o(\varphi_o)}, \quad (\varphi, P) \mapsto \left\{ (\gamma, t) \in \Gamma \times \bR \mid \varphi(\gamma) + t \in P \right\}\end{equation}
is a well-defined $(\Gamma \times \bR)$-equivariant homeomorphism.
\end{proposition}
\begin{proof} By \eqref{CanonicalCocycle} the map
\[
c: \Gamma \to \Delta_{\phi_o} \to \R, \quad (\gamma, \varphi) \mapsto \varphi(\gamma).
\]
is a continuous normalized cocycle. This cocycle separates points, since for all $\varphi, \varphi' \in \Delta_{\varphi_o}$ we have
\[
\varphi(\gamma) = c(\gamma,\varphi) = c(\gamma,\varphi') = \varphi'(\gamma) \enskip \textrm{for all $\gamma \in {\Gamma}$} \implies \varphi = \varphi'.
\]
Since $y_o := \varphi_o$ has a dense orbit in $\Delta_{\varphi_o}$, we can thus apply Corollary \ref{Critb} with $y := \varphi_o$. The proposition then follows from the fact 
\begin{align*}
 \pi(\varphi_o, P_o) &= \{(\gamma,t) \in \Gamma \times \R \mid c(\gamma, \varphi_o)+t\in P_o \}\\
 &= \{(\gamma,t) \in \Gamma \times \R  \mid \varphi_o(\gamma) + t \in P_o\}  =  P_o(\varphi_o).\qedhere
\end{align*}
\end{proof}

\subsection{Stationary and invariant measures on skew products}
Let $c: \Gamma \times \Delta \to \bR$ be a normalized  continuous cocycle and let $Z$ be a compact $\R$-space. We investigate stationary and invariant measures on the skew-product $\Delta \times_c Z$. The following lemma is an immediate consequence of Remark \ref{BaseFactor}:
\begin{lemma}\label{NoInvMeasure} If $\mathrm{Prob}_\Gamma(\Delta) = \emptyset$, then $\mathrm{Prob}_{\Gamma \times \R}(\Delta \times_c Z) = \emptyset$.\qed
\end{lemma}
If $Z$ admits an $\R$-invariant probability measure $\theta$, then we can relate stationary (in particular, invariant) measures on $\Delta$ on $\Delta \times_c Z$. To make this precise, we fix a symmetric probability measure $\widetilde{p}$ on $\Gamma \times \R$; we are going to assume that $\widetilde{p}$ is absolutely continuous with respect to the Haar measure on $\Gamma \times \R$ and that it supports generates $\Gamma \times \R$ as a semigroup. We then denote by $p \in \mathrm{Prob}(\Gamma)$ the pushforward of $\widetilde{p}$ to $\Gamma$. We then obtain a well-defined injection
\[
\sigma_\theta : \Prob_p(\Delta) \ra \Prob_{\widetilde{p}}(\Delta \times_c Z), \enskip \nu \mapsto \nu \otimes \theta
\]
\begin{lemma}\label{StationaryDown} If $Z$ is uniquely ergodic with $\mathrm{Prob}_\bR(Z) = \{\theta\}$, then $\sigma_\theta$ is onto, i.e.\ every $\widetilde{p}$-stationary probability measure on $\Delta \times_c Z$ is of the form $\nu \otimes \theta$ for some unique $\nu \in  \Prob_p(\Delta)$.
\end{lemma}
\begin{proof} Let $\mu$ be a $\widetilde{p}$-stationary probability measure on $\Delta \times_c Z$, and let $p_\bR$ denote the push-forward of $p$ to $\bR$. Then 
\[
(\delta_e \otimes p_\bR) * \widetilde{p} = \widetilde{p} * (\delta_e \otimes p_\bR).
\]
so by \cite[Corollary 3.2]{BHO}, $\mu$ is also $(\delta_e \otimes p_\bR)$-stationary. Fix a Borel set $B \subset Y$ and define the positive Borel measure $\mu_B$ on $Z$ by
\[
\mu_B(C) = \mu(B \times C), \quad C \in \mathscr{B}_Z.
\]
Since $\mu$ is $(\delta_e \otimes p_\bR)$-stationary, we conclude that $\mu_B$ is $p_\bR$-stationary, and thus $\bR$-invariant by the Choquet-Deny Theorem \cite[Theorem 3]{Blackwell}. Since $\Prob_\bR(Z) = \{\theta\}$, we conclude that there exists a non-negative constant $\nu(B)$
such that 
\[
\mu_B(C) = \mu(B \times C) = \nu(B) \theta(C).
\]
One readily checks that $B \mapsto \nu(B)$ is a Borel probability measure on $\Delta$, which is the push-forward of $\mu$ under the canonical projection $Y \times_c Z \ra Y$. 
Since $\mu$ is $\widetilde{p}$-stationary, $\nu$ is $p$-stationary.
\end{proof}

\subsection{Stationary and invariant measures on hulls of twisted sets}
Combining Proposition \ref{SkewRealization} with Proposition \ref{QMTrivialities}.(iii), Corollary \ref{LeftHarmonicCase}, Remark \ref{BaseFactor}, Lemma \ref{NoInvMeasure} and Lemma \ref{StationaryDown} we obtain:
\begin{theorem}\label{FinalThm} In the situation of Setting \ref{MainSetting} the following hold:
\begin{enumerate}[(i)]
\item If $\varphi_o$ is cohomologically non-trivial, then $\mathrm{Prob}_{\Gamma \times \R}(\Omega_{{P}_o(\varphi_o)}) = \emptyset$.
\item There is a continuous $\Gamma$-equivariant factor map \[\Omega_{{P}_o(\varphi_o)} \xrightarrow{\pi_{P_o, \phi_o}^{-1}}  \Delta_{\varphi_o} \times_c \Omega_{P_o} \to \Delta_{\varphi_o}.\]
\item If $\Omega_{P_o}$ is uniquely ergodic with $\R$-invariant probability measure $\theta$, then the map
\[
\pi_{P_o, \phi_o}^\theta: \mathrm{Prob}_p(\Delta_{\phi_o}) \to \mathrm{Prob}_{\widetilde{p}}(\Omega_{P_o(\phi_o)}), \quad \nu \mapsto (\pi_{P_o, \phi_o})_*(\nu \otimes \theta) \]
is a homeomorphism.
\item If $\Omega_{P_o}$ is uniquely ergodic and $\phi_o$ is left-$p$-harmonic, then $\Omega_{P_o(\varphi_o)}$ is uniquely $\widetilde{p}$-stationary and an isometric extension of a $p$-boundary (which is non-trivial if $\phi_o$ is cohomologically non-trivial).\qed
\end{enumerate}
\end{theorem}
\begin{remark}[Classification of stationary measures]
If $\Omega_{P_o}$ is uniquely ergodic, but $\phi_o$ is not left-$p$-harmonic, then the classification of $\widetilde{p}$-stationary measures on $\Omega_{P_o(\varphi_o)}$ is more complicated. In this situation we can argue as follows:

Let $\mu$ be a $\widetilde{p}$-stationary measure on $\Omega_{P_o(\phi_o)}$;  by Theorem \ref{FinalThm} we have $\mu =  (\pi_{P_o, \phi_o})_*(\nu \otimes \theta)$ for some (unique) $\nu \in \mathrm{Prob}_p(\Delta_{\phi_o})$. Now let $\psi_o \in \mathcal L_p(\phi_o)$ and denote by $\nu'$ the unique $p$-stationary measure on $\Delta_{\psi_o}$. By Corollary \ref{CorollaryMain} there is a $\Gamma$-equivariant $\nu$-measurable map $q: (\Delta_{\phi_o}, \nu) \to (\Delta_{\psi_o}, \nu')$ (which coincides with a Borel map on a $\nu$-conull set) such that $q_\ast \nu = \nu'$. We now consider the set $P_o(\psi_o)$, which by Lemma \ref{BoundedDisplacement} is bounded displacement equivalent to $P_o(\phi_o)$.
If we now define $\mu' =  (\pi_{P_o, \psi_o})_*(\nu' \otimes \theta)$, then we obtain a commuting diagram of measurable (but not Borel) maps
\begin{equation}
\begin{xy}\xymatrix{
(\Omega_{P_o(\varphi_o)}, \mu) \ar[rr] \ar[d] && (\Omega_{P_o(\psi_o)}, \mu') \ar[d]\\
(\Delta_{\phi_o}, \nu) \ar[rr] && (\Delta_{\psi_o}, \nu')
}\end{xy}
\end{equation}
with the following properties:
\begin{enumerate}[(1)]
\item All maps are measure-preserving and all measures are stationary.
\item The vertical maps are continuous factor maps, in fact, isometric extensions.
\item The spaces on the right are uniquely stationary with unique stationary measures $\mu'$, respectively $\nu'$.
\item $(\Delta(\psi_o), \nu')$ is a $p$-boundary.
\end{enumerate}
In general, the spaces on the left are \emph{not} uniquely stationary. In fact, by Theorem \ref{ManyStationaryMeasures} we can always modify $\psi_o$ by adding a bounded function and achieve in this way that
\[
| \mathrm{Prob}^{\textrm{erg}}_{\widetilde{p}}(\Omega_{P_o(\varphi_o)})| = | \mathrm{Prob}^{\textrm{erg}}_p(\Delta_{\phi_o})| = 2^{\aleph_o}.
\]
\end{remark}
\begin{example} Let $\alpha$ be irrational and consider the subset
\[
P_o := \{m + n\alpha \mid m, n \in \bZ, m - n\alpha \in [-1,1]\}.
\]
Then $P_o \subset \bR$ is a uniform approximate lattice such that the $\R$-action on $\Omega_{P_o}$ is uniquely ergodic and free. Moreover, let $F$ be a non-abelian free group with basis S and consider
\[\widetilde{p} := (\sum_{s \in S \cup S^{-1}}\delta_s) \otimes \lambda|_{[-1/2,1/2]} \in \Prob(F \times \bR) \qand p := \sum_{s \in S \cup S^{-1}}\delta_s \in \Prob(F).\]
We now define three subsets of $F \times \bR$ as follows: Firstly, let $\varphi_o: F \to \bZ$ be a counting quasimorphism and $Q_1 := P_o(\varphi_o)$. Secondly, let $\psi_o$ be the unique $p$-biharmonic quasimorphism equivalent to $\varphi_o$ and set $Q_2 := P_o(\psi_o)$. Finally, let $\phi_1: F \to \bZ$ be an antisymmetric quasimorphism equivalent to $\varphi_o$ such that $|\mathrm{Prob}^{\mathrm{erg}}_p(\Delta_{\varphi_1})| = 2^{\aleph_o}$ (cf.\ Theorem  \ref{ManyStationaryMeasures}) and let $Q_3 := P_o(\varphi_1)$. Then the following hold:
\begin{enumerate}
\item $Q_1$ and $Q_3$ are uniform approximate lattices in $F \times \bR$.
\item The hulls of $Q_1$, $Q_2$ and $Q_3$ do not admit any $F \times \bR$-invariant probability measures.
\item $\Omega_{Q_2}$ admits a unique $p$-stationary measure $\nu_2$ and $(\Omega_{Q_2}, \nu_2)$ is a relative isometric extension of a non-trivial $p$-boundary.
\item For every $p$-stationary probability measure $\nu_1$ on $\Omega_{Q_1}$ there is a measure-preserving measurable $(F \times \bR)$-equivariant map $(\Omega_{Q_1}, \nu_1) \to (\Omega_{Q_2}, \nu_2)$.
\item $\Omega_{Q_3}$ admits $2^{\aleph_o}$ different $p$-stationary $(F \times \bR)$-ergodic probability measures.
\end{enumerate}
\end{example}

\end{document}